\documentclass[11pt,oneside]{article}

\usepackage{mathptmx}

\usepackage{amsmath}
\usepackage{mathtools}
\usepackage{psfrag}
\usepackage[english]{babel} 
\usepackage{graphicx}
\usepackage{caption}
\usepackage{amssymb}
\usepackage{verbatim, bbm}
\usepackage{tabu}
\usepackage{amsbsy,bm}
\usepackage{dsfont}

\usepackage{makecell}
\usepackage{colortbl}
\usepackage{multirow}
\usepackage{hhline}
\usepackage{arydshln}
\usepackage{xspace}

\usepackage{enumerate}
\usepackage{amsfonts}
\usepackage{color}
\usepackage[overload]{empheq}
\usepackage{subfig}
\usepackage{float}
\usepackage[section]{placeins}
\usepackage[hyphens]{url}
\usepackage{hyperref}
\usepackage{xcolor}
\hypersetup{
	colorlinks,
	linkcolor={red!50!black},
	citecolor={blue!50!black},
	urlcolor={blue!80!black}
}
\usepackage[nosort]{cite}
\usepackage{mathrsfs}

\usepackage{tikz}
\usetikzlibrary{arrows,shapes}
\usetikzlibrary{arrows.meta}

\newtheorem{proposition}{Proposition}[section]
\newtheorem{theorem}{Theorem}[section]
\newtheorem{lemma}[proposition]{Lemma}
\newtheorem{corollary}[proposition]{Corollary}
\newtheorem{definition}{Definition}[section]
\newtheorem{remark}{Remark}

\newcommand\restr[2]{{
		\left.\kern-\nulldelimiterspace 
		#1 
		\vphantom{\big|}
		\right|_{#2} 
}}

\newenvironment{proof}{{\noindent \it Proof.}}{\hfill $\fbox{}$ \vspace*{5mm}}

\newcommand{\N}{\mathbb{N}}
\newcommand{\R}{\mathbb{R}}

\newcommand{\diag}{\mathop{\mathrm{diag}}}

\newcommand{\Eulerc}{\mathcal{L}_{\textnormal{dir},\alpha x^2}}
\newcommand{\Euler}{\mathcal{L}_{\textnormal{dir},\alpha x^2}^{(n)}}
\newcommand{\Eulerw}{\hat{\mathcal{L}}_{\textnormal{dir},\alpha x^2}^{(n)}}

\newcommand{\numerr}{\prescript{}{\alpha}{\textnormal{\textbf{err}}_{k}^{(n)}}}
\newcommand{\analerr}{\prescript{}{\alpha}{\textnormal{\textbf{err}}_{r,k}^{*,(n)}}}
\newcommand{\analerrExact}{\prescript{}{\alpha}{\textnormal{\textbf{err}}_{k}^{*,(n)}}}

\newcommand{\bfx}{\mathbf x}

\newcommand{\bfnn}{{\boldsymbol n}}\newcommand{\bfxx}{{\boldsymbol x}}\newcommand{\bfyy}{{\boldsymbol y}}
\newcommand{\bftheta}{{\boldsymbol\theta}}

\DeclareMathSymbol{\shortminus}{\mathbin}{AMSa}{"39}

\numberwithin{equation}{section}
\date{}

\begin{document}
	\title{Analysis of the spectral symbol associated to discretization schemes of linear self-adjoint differential operators\footnote{This is a preprint.}}
	\author{Davide Bianchi\footnote{University of Insubria - Como (Italy). Email: d.bianchi9@uninsubria.it} \textsuperscript{ ,}\footnote{Partially supported by INdAM-GNCS Gruppo Nazionale per il Calcolo Scientifico}}
	\maketitle
	
	\begin{abstract}
Given a linear self-adjoint differential operator $\mathcal{L}$ along with a discretization scheme (like Finite Differences, Finite Elements, Galerkin Isogeometric Analysis, etc.), in many numerical applications it is crucial to understand how good the (relative) approximation of the whole spectrum of the discretized operator $\mathcal{L}^{(\bfnn)}$ is, compared to the spectrum of the continuous operator $\mathcal{L}$.  The theory of Generalized Locally Toeplitz sequences allows to compute the spectral symbol function $\omega$ associated to the discrete matrix $\mathcal{L}^{(\bfnn)}$. 

We prove that the symbol $\omega$ can measure, asymptotically, the maximum spectral relative error $\mathcal{E}\geq 0$. It measures how the scheme is far from a good relative approximation of the whole spectrum of $\mathcal{L}$, and it suggests a suitable (possibly non-uniform) grid such that, if coupled to an increasing refinement of the order of accuracy of the scheme, guarantees $\mathcal{E}=0$. 
\end{abstract}

\section{Introduction}\label{sec:intro}

Denoting by $\mathcal{L}^{(\bfnn)}$ the $d_\bfnn\times d_\bfnn$ matrix that discretizes a linear self-adjoint differential operator $\mathcal{L}$, obtained by a discretization scheme like Finite Differences (FD), Finite Elements (FE), Isogeometric Galerkin Analysis (IgA), etc., several problems require that the spectrum of $\mathcal{L}^{(\bfnn)}$ converges to the (point) spectrum of the operator $\mathcal{L}$ uniformly with respect to $k=-d^-_\bfnn,\cdots,d^+_\bfnn$ (where $d^+_\bfnn$ is the sum of the null and the positive indices of inertia of $\mathcal{L}^{(\bfnn)}$ and $d^-_\bfnn$ is the negative index of inertia of $\mathcal{L}^{(\bfnn)}$), i.e., that
\begin{equation}\label{max_rel_spectrum}
\mathcal{E}:= \limsup_{\bfnn \to \infty}\mathcal{E}_\bfnn=0, \qquad  \mathcal{E}_\bfnn:=\max_{\substack{k=-d^-_\bfnn,\ldots, d^+_\bfnn\\k\neq 0}}\left|\frac{\lambda_k^{(\bfnn)}}{\lambda_k}-1\right|, 
\end{equation}
 where $\bfnn$ is the mesh finesses parameter, and $\lambda_{k}^{(\bfnn)}$ and $\lambda_{k}$ are the $k$-th positive/negative eigenvalues of the discrete and the continuous operator, respectively, sorted in increasing order. Typically, $\mathcal{E}>0$: the relative error estimates for the eigenvalues and eigenfunctions are good only in the \emph{lowest modes}, that is, for $|k|=1,\ldots, K_\bfnn$, where $K_\bfnn/d_\bfnn\to\sigma$ as $\bfnn\to \infty$ and such that $\sigma <<1$. In general, a large portion of the eigenvalues, the so-called \emph{higher modes}, are not approximations of their continuous counterparts in any meaningful sense. This may negatively affects the solutions obtained by discrete approximations of elliptic boundary-value problems, or parabolic and
 hyperbolic initial value problems. In these cases, all modes may contribute in the solution to some extent and inaccuracies in the approximation of higher modes can not always be ignored.
 
 Regarding this, see for example the spectral-gap problem of the one dimensional wave equation for the uniform observability of the control waves, \cite{IZ99,MZ15,EMZ16,BS18}, or structural engineering problems, see \cite[sections 3--6]{HER}.

In this setting, the theory of Generalized Locally Toeplitz (GLT) sequences provides the necessary tools to understand whether the methods used to discretize the operator $\mathcal{L}$ are effective in approximating the whole spectrum. The GLT theory originated from the seminal work of P. Tilli on Locally Toeplitz sequences \cite{Ti98} and later developed by S. Serra-Capizzano in \cite{Serra03,Serra06}. It was devised to compute and analyze the spectral distribution of matrices arising from the numerical discretization of integral equations and differential equations. 

Let $\hat{\mathcal{L}}^{(\bfnn)}$ be the discrete operator suitably weighted by a power of $d_\bfnn$, which depends on the dimension of the underlying space and on the maximum order of derivatives involved. It usually happens that the sequence of matrices $\left\{ \hat{\mathcal{L}}^{(\bfnn)} \right\}_{\bfnn}$ enjoys an asymptotic spectral distribution as $\bfnn\to \infty$, i.e., as the mesh goes to zero. More precisely, for any test function $F(t) \in C_c(\mathbb{C})$ it holds that
 \begin{equation*}\label{eq:spectral_symbol_intro}
 \lim_{\bfnn \to \infty} \frac{1}{d_\bfnn} \sum_{\substack{k=-d^-_\bfnn\\k\neq 0}}^{d^+_\bfnn}F\left(\lambda_k\left(\hat{\mathcal{L}}^{(\bfnn)}\right)\right) = \frac1{m(D)}\int_D  F(\omega(\boldsymbol{y}))m(d\boldsymbol{y}),
 \end{equation*}
where $\lambda_k\left(\hat{\mathcal{L}}^{(\bfnn)}\right), k=-d^-_\bfnn,\ldots,d^+_\bfnn$ are the eigenvalues of the weighted operator $\hat{\mathcal{L}}^{(\bfnn)}$ and $\omega: D\subset \R^M \to \mathbb{C}$ is referred to as the {\em spectral symbol} of the sequence $\left\{\hat{\mathcal{L}}^{(\bfnn)} \right\}_{\bfnn}$, see \cite[Equation (3.2)]{Tyrty96} in relation to Toeplitz matrices. 

The GLT theory allows to compute the spectral symbol $\omega$ related to $\mathcal{L}^{(\bfnn)}$, especially if the numerical method employed to produce $\mathcal{L}^{(\bfnn)}$ belongs to the family of the so-called {\em local methods}, such as FD methods,
FE methods and collocation methods  with locally supported basis functions.

The symbol $\omega$ can measure the maximum spectral relative error $\mathcal{E}$ defined in \eqref{max_rel_spectrum} and it suggests a suitable (non-uniform) grid such that, if coupled to an increasing refinement of the order of accuracy of the scheme, guarantees $\mathcal{E}=0$.

Moreover, in several recent papers the sampling of the spectral symbol was suggested to be used to approximate the spectrum of the discrete matrix operators $\hat{\mathcal{L}}^{(\bfnn)}$ and $\mathcal{L}^{(\bfnn)}$. Unfortunately, this approach is not always successful in general, as we will show with an example. The main reference is the paper \cite{GSERSH18} which reviews the state-of-the-art of the symbol-based analysis for the eigenvalues distribution carried on in the framework of the isogeometric Galerkin approximation (IgA).

The paper is organized as follows. 
\begin{itemize}

	\item In Section \ref{sec:GLT}, the spectral symbol $\omega$ and the monotone rearrangement $\omega^*$ are introduced. In Section \ref{sec:asymptotic_distribution} we prove an asymptotic result, Theorem \ref{thm:discrete_Weyl_law}, which connects the eigenvalue distribution of a matrix sequence and the monotone rearrangement of its spectral symbol. It is the main tool for the results in Section \ref{sec:example} and Appendix \ref{sec:theory}. In particular, under suitable regularity assumptions on $\omega$, we prove that 
\begin{equation}\label{eq:asymptotic_comparison}
\mathcal{E}=\sup_{\substack{x\in(0,1)\\x:\zeta^{*}(x)\neq0}}\left| \frac{\omega^{*}(x)}{\zeta^{*}(x)} -1\right|,
\end{equation}
	where $\zeta$ is the (weighted) Weyl distribution function of the eigenvalues of $\mathcal{L}$ and $\omega$ is the spectral symbol associated to the numerical scheme applied to discretize $\mathcal{L}$, see Theorem \ref{thm:necessary_cond_for_uniformity} and Section \ref{ssec:application_to_self-adjoint_operators}.
	
	\item Section \ref{sec:example} is devoted to numerical experiments: the validity of \eqref{eq:asymptotic_comparison} is shown in Table \ref{table:maximum_rel_error} and Figure \ref{fig:eig_symbol_comparison}. Moreover, in Subsection \ref{ssec:example_uniform_3_points}  we provide an example about the unfeasibility to obtain an accurate approximation of the eigenvalues of a differential operator by just uniformly sampling the spectral symbol $\omega$.
	
	In Subsection \ref{ssec:FD_nonuniform} and Subsection \ref{ssec:galerkin} we generalize the results in the previous subsection to the case of central FD and IgA methods of higher order. By means of a suitable non-uniform grid, suggested by the spectral symbol and combined by an increasing order of the approximation, we obtain \eqref{max_rel_spectrum}, namely $\mathcal{E}=0$ .  
	
	In Subsection \ref{ssec:generalization}, we describe an application of our results to a class of linear self-adjoint elliptic differential operators on bounded domains.	
\end{itemize}


\subsection{About the notations}\label{ssec:notations}
We will write in bold all multi-dimensional variables and vector-valued functions. Given an integer $d\geq 1$, a $d$-index $\bfnn$ is an element of $\mathbb{N}^d$, that is, $\bfnn=\left(n_1, \ldots, n_d\right)$ with $n_j \in \mathbb{N}$ for every $j=1,\ldots, d$. Throughout this paper $\mathbb{N}^d$ will be endowed with the lexicographic ordering. We write $\bfnn \to \infty$ meaning that $\min_{j=1,\ldots,d}\{n_j\}\to \infty$.
	
 Given a $d$-index $\bfnn$, we let $d_\bfnn:= \prod_{j=1}^d n_j$. We will use the notation $A^{(\bfnn)},B^{(\bfnn)},\ldots$ to denote general $d_\bfnn\times d_\bfnn$ square matrices. When we will confront the spectrum of two sequences of matrices, we will assume that they have the same dimension. In the case $d=1$, $T^{(n)}$ will denote a Toeplitz matrix, i.e., a matrix with constant coefficients along its diagonals: $\left(T^{(n)}\right)_{i,j}= t_{i-j}$ for all $i,j=1,\ldots,d_n$, and with $\mathbf{t}=\left[t_{-d_n+1}, \ldots, t_0, \ldots, t_{d_n-1}\right] \in \mathbb{C}^{2d_n-1}$. If $t_k$ is the $k$-th Fourier coefficient of a complex integrable function $f$ defined over the interval $[-\pi, \pi]$, then $T^{(n)}= T^{(n)}(f)$ is referred to be the Toeplitz matrix generated by $f$.
 
 We will write $\mathcal{L}^{(\bfnn,\eta)}$ to denote a $d_\bfnn\times d_\bfnn$ square matrix which is the discretization of a linear differential operator $\mathcal{L}$ by means of a numerical scheme of order of approximation $\eta$. In the case where the approximation order $\eta$ is clear by the context, then we will omit it. If the discretized operator $\mathcal{L}^{(\bfnn)}$ is weighted by a constant depending on the finesses mesh parameter $\bfnn$, then we will denote it with $\hat{\mathcal{L}}^{(\bfnn)}$. We will use the subscripts {\em dir} and {\em BCs} to denote a (discretized) linear differential operator characterized with Dirichlet or generic boundary conditions, respectively. When it will be necessary to highlight the dependency of the differential operator on a variable coefficient $p(\bfxx)$, we will write it as subscript. So, for example, the weighted discretization of a diffusion operator with Dirichlet BCs by means of the IgA scheme of order $\eta$ will be denoted by 
 $$
 \hat{\mathcal{L}}^{(\bfnn,\eta)}_{\textnormal{dir},p(\bfxx)}.
 $$
 In the special case of the (negative) Laplace operator we will use the symbol $\Delta:= -\sum_{i=1}^d \partial^2_{x^2_i}$, and all the previous notation will apply.
 
 We will consider all the Euclidean spaces equipped with the usual Lebesgue measure $m(\cdot)$. We will not specify the dimension of $m(\cdot)$ by a subscript since it will be clear from the context.  
  
 We will use the letter $c$ for all the constants, making explicit the dependency on other parameters if needed.
 
 Finally, for any fixed $k,j \in \mathbb{N}\setminus\{0\}$ we will denote with $\lambda_{k}$ and $\lambda_{-j}$ the $k$-th non-negative and the $j$-th negative real eigenvalues of a given operator, respectively. Given a $d_\bfnn\times d_\bfnn$ matrix $X^{(\bfnn)}$ with $d_\bfnn$ real eigenvalues, we will denote with $d^+_\bfnn$ the sum of the null and positive indices of inertia  and with $d^-_\bfnn$ the negative index of inertia, i.e., 
 $$
 d^+_\bfnn:= \left|\left\{\lambda\left(X^{(\bfnn)}\right) \, : \, \lambda\left(X^{(\bfnn)}\right)\geq 0  \right\}  \right|, \qquad d^-_\bfnn:= \left|\left\{\lambda\left(X^{(\bfnn)}\right) \, : \, \lambda\left(X^{(\bfnn)}\right)< 0  \right\}  \right|.
 $$
  Clearly, $d_\bfnn = d^+_\bfnn + d^-_\bfnn$. The eigenvalues will be sorted in non-decreasing order, that is
  $$
  \ldots \leq \lambda_{-j} \leq \ldots \leq \lambda_{-1}< 0 \leq \lambda_1 \leq \ldots \lambda_k \leq \ldots 
  $$

\section{Spectral symbol and monotone rearrangement}\label{sec:GLT}
In this section we provide the definitions of spectral symbol of a sequence of matrices and its monotone rearrangement, which are the main tools we will use throughout this paper to study the asymptotic spectral distribution of sequences of discretization matrices.  
\subsection{Spectral symbol}

Hereafter, with the symbol $\left\{X^{(\bfnn)}\right\}_\bfnn$ we will denote a sequence of square matrices with increasing dimensions $d_\bfnn\times d_\bfnn$, i.e., such that $d_\bfnn \to \infty$ as $\bfnn \to \infty$. The following definition of spectral symbol has been slightly modified in accordance to our notation and purposes.

\begin{definition}[Spectral symbol]\label{def:ss_def}
	Let $\left\{X^{(\bfnn)}\right\}_\bfnn$ be a sequence of matrices and let $\omega:D\subset \R^M\to\mathbb{K}$ ($\mathbb{K}=\R$ or $\mathbb{C}$) be a Borel-measurable function and $D$ a Borel with $0<m(D)<\infty$, such that the Lebesgue integral 
	$$
	\iint_{D} \omega(\bfyy)m(d\bfyy)
	$$
exists, finite or infinite. We say that $\{X^{(\bfnn)}\}_\bfnn$ is distributed like $\omega$ in the sense of the
	eigenvalues, in symbols $\{X^{(\bfnn)}\}_\bfnn \sim_\lambda \omega$, if
	\begin{equation}\label{def_asym-bis-Matrix}
\lim_{\bfnn\rightarrow \infty}\frac{1}{d_\bfnn}\sum_{\substack{k=-d^-_\bfnn\\k\neq 0}}^{d^+_\bfnn}F\left(\lambda_k\left(X^{(\bfnn)}\right)\right)=\frac{1}{m(D)}\iint_{D}  F\left(\omega(\bfyy)\right)m(d\bfyy),\qquad\forall F\in C_c(\mathbb K).
	\end{equation}
	We will call $\omega$ the \textnormal{(spectral) symbol} of $\{X^{(\bfnn)}\}_\bfnn$.
\end{definition}

Relation \eqref{def_asym-bis-Matrix} is satisfied for example by Hermitian Toeplitz matrices generated by real-valued functions $\omega\in \textnormal{L}^1([-\pi,\pi])$, i.e., $\left\{T^{(n)}(\omega) \right\}_n\sim_\lambda \omega$, see \cite{Szego84,Tyrty96,Tyrty98}. For a general overview on Toeplitz operators and spectral symbol, see \cite{BG00,BS13}. What is interesting to highlight is that matrices with a Toeplitz-like structure naturally arise when discretizing, over a uniform grid, problems which have a translation invariance property, such as linear differential operators with constant coefficients.

\begin{remark}\label{rem:ssymbol_sampling}
In particular, if $D$ is compact, $\omega$ continuous, and $\lambda_k \in R_\omega= \left[\min \omega, \max \omega\right]$ for every $k\in \N$, then 
 taking $F(t)=t\chi_{R_\omega}(t)$, with $\chi_{R_\omega}$ a $C^\infty$ cut-off such that $\chi_{R_\omega}(t)\equiv 1$ on $R_\omega$, it holds that
	\begin{equation}\label{ss_def2}
	\lim_{\bfnn \to \infty} \frac{1}{d_\bfnn} \sum_{\substack{k=-d^-_\bfnn\\k\neq 0}}^{d^+_\bfnn}\lambda_k\left(X^{(\bfnn)}\right) = \frac{1}{m(D)}\iint_{D} \omega(\bfyy)m(d\bfyy).
	\end{equation}
	Because the Riemannian sum over equispaced points converges to the integral of the right hand side of the above formula, then \eqref{ss_def2} could suggest that the eigenvalues $\lambda_k\left(X^{(\bfnn)}\right)$ can be approximated by a pointwise evaluation of the symbol $\omega(\bfyy)$ over an equispaced grid of $D$, for $\bfnn\to \infty$, expect for at most an $o(d_\bfnn)$ of outliers, see Definition \ref{def:outliers}. This is mostly the content of \cite[Remark 3.2]{GS17} and \cite[Section 2.2]{GSERSH18}. 

We remark that in the special cases of Toeplitz matrices generated by enough regular functions $\omega$, 
$$
\lambda_k\left(T^{(n)}(\omega)\right)= \omega\left(\frac{k\pi}{d_n+1}\right) + O(d_n^{-1}) \qquad \mbox{for every } k=1,\cdots,d_n,
$$
see \cite{W58} and \cite{BBGM15}. In some sense, this justifies the informal meaning given above to the spectral symbol.
\end{remark}

That said, unfortunately the discretization of a linear differential operator does not always own a Toeplitz-like structure. Nevertheless, the GLT theory provides tools to prove the validity of \eqref{def_asym-bis-Matrix} for more general matrix sequences. For a review of the GLT theory we mainly refer to \cite{GS17,GS18} and all the references therein. We conclude this subsection with two definitions.

\begin{definition}[Essential range]\label{def:essential_range}
	Let $\omega$ be a real valued measurable function and define the set $R_\omega \subset \mathbb{R}$ as 
	\begin{equation*}
	t\in R_\omega \quad \Leftrightarrow \quad m\left(\left\{\boldsymbol{y} \in D \, : \, \left| \omega(\boldsymbol{y}) - t\right|< \epsilon   \right\}\right)>0 \quad \forall \epsilon >0.
	\end{equation*}
	We call $R_\omega$ the \emph{essential range} of $\omega$.  $R_\omega$ is closed.
\end{definition}

\begin{definition}[Outliers]\label{def:outliers}
	Given  a matrix sequence $\{X^{(\bfnn)}\}$ such that $\left\{X^{(\bfnn)}\right\}_{\bfnn} \sim_{\lambda} \omega$, if $\lambda_k\left( X^{(\bfnn)}\right)\notin R_\omega$ we call it an \emph{outlier}.
\end{definition}

\subsection{Monotone rearrangement}\label{ss:rearrangment}

Dealing with an univariate and monotone spectral symbol $\omega(\bfyy) = \omega(\theta)$ has several advantages. Unfortunately, in general $\omega$ is multivariate or not monotone. Nevertheless, it is possible to consider a rearrangement $\omega^*: (0,1) \to(\inf R_\omega, \sup R_\omega)$ such that $\omega^*$ is univariate, monotone nondecreasing and still satisfies the limit relation \eqref{def_asym-bis-Matrix}. This can be achieved in the following way.
\begin{definition} 
	Let $\omega : D \subset \R^M \to \R$ be Borel-measurable. Define
\begin{equation}\label{eq:rearrangment}
\omega^* :  (0,1)\to R_\omega, \qquad\omega^*(x) = \inf\left\{ t \in R_\omega\,:\, \phi_\omega(t)> x\right\} 
\end{equation}
where 
\begin{equation}\label{eq:rearrangment2}
\phi_\omega : \R \to [0,1], \qquad \phi_\omega(t) := \frac{1}{m(D)}m \left(\left\{\bfyy \in D \, : \, \omega(\bfyy) \leq t  \right\}\right),
\end{equation}
and where, in case of bounded $R_\omega$, we consider the extension $\omega^*: [0,1] \to R_\omega$,  see Theorem \ref{thm:discrete_Weyl_law}.
\end{definition}
In Analysis, $\omega^*$ is called \emph{monotone rearrangement} (see \cite[pg. 189]{Stein71}) while in Probability Theory it is called (generalized) inverse distribution function (see \cite[pg. 260]{Taylor97}). For \textquotedblleft historical\textquotedblright reasons we will use the analysts' name, see \cite[Definition 3.1 and Theorem 3.3]{DBFS93} and \cite{Serra98}, where the monotone rearrangement were first introduced in the context of spectral symbol functions.  Clearly, $\omega^*$ is a.e. unique, univariate, and monotone increasing, which make it a good choice for an equispaced sampling. On the other hand, $\omega^*$ could not have an analytical expression or it could be not feasible to compute, therefore it is often needed an approximation $\omega^*_r$. Hereafter we summarize the steps presented in \cite[Example 10.2]{GS17} and \cite[Section 3]{GSERSH18} in order to approximate the eigenvalues $\lambda_k\left(T^{(n)}\right)$ by means of an equispaced sampling of the rearranged symbol $\omega^*$ (or its approximated version $\omega^*_r$). For the sake of clarity and without loss of generality, we suppose $d=1$, $D=[0,1]\times [-\pi,\pi]$ and $\omega$ continuous.
\vspace{0.3cm}

\noindent\textbf{Algorithm}
	\begin{enumerate}[1)]
		\item Fix $r \in \N$ such that $r=r(n)\geq d_n$, and fix the equispaced grid $\{(x_i, \theta_j)\}$ over $[0,1]\times[-\pi,\pi]$, where $x_i= \frac{i}{r+1}$, $\theta_j = -\pi + \frac{j2\pi}{r+1}$ for $i,j=1,\cdots, r$;  
		\item Get the set of samplings $\left\{ \omega(x_i,\theta_j) \right\}_{i,j=1}^r$ and form a nondecreasing sequence $\{\omega_1\leq \omega_2\leq \cdots\leq \omega_{r^2} \}$;
		\item Define $\omega^*_r :[0,1] \to [\min\omega, \max\omega]$ as the piecewise linear nondecreasing function which interpolates the samples $\{\min\omega=\omega_0\leq\omega_1\leq \omega_2\leq \cdots\leq \omega_{r^2}\leq\omega_{r^2+1}=\max\omega \}$ over the nodes $\{0,\frac{1}{r^2+1},\frac{2}{r^2+1},\cdots,\frac{r^2}{r^2+1},1\}$;
		\item Sample $\omega^*_r$ over the set $\left\{ \frac{k}{d_n+1}\right\}_{k=1}^{d_n}$ and define $\omega^{*,(n)}_{r,k}:= \omega^*_r\left( \frac{k}{d_n+1}\right)$. Obviously, if $\omega^*$ is available then use it instead of $\omega^*_r$ and define $
		\omega^{*,(n)}_{k}:= \omega^*\left( \frac{k}{d_n+1}\right)$.
	\end{enumerate}

As standard result in approximations of monotone rearrangements, it holds that $\| \omega^*_{r(n)} - \omega^* \|_\infty \to 0$ as $n\to \infty$, see \cite{CP79,T86}. 

\section{Asymptotic spectral distribution}\label{sec:asymptotic_distribution}
In this section we prove one of the main results about the asymptotic spectral distribution of the eigenvalues of a matrix sequence with given spectral symbol. 

\begin{definition}[Vague convergence]\label{def:vague_conv}
A measure $\mu(\cdot)$ on $\left(\R, \mathcal{B}\right)$, with $\mathcal{B}$ the Borel set of $\R$, is said a \emph{subprobability measure} (s.p.m) if $\mu(\R)\leq 1$. A sequence $\{\mu_n\}_n$ of s.p.m is said to \emph{converge vaguely} to a s.p.m. $\mu$ iff for every $a<b$ such that $\mu((a,b))=\mu([a,b])$, then
$$
\lim_{n \to \infty}\mu_n((a,b]) = \mu((a,b]),
$$ 	
and we write $\mu_n \xrightarrow{v} \mu$. See the definition given in \cite[p. 85 and Theorem 4.3.1]{Chung01}.
\end{definition}

It holds the following result (\cite[Theorem 4.4.1]{Chung01}).
\begin{proposition}\label{prop:vague_conv}
Let $\{\mu_n\}_n$ and $\mu$ be s.p.m. Then $\mu_n \xrightarrow{v} \mu$ iff 
\begin{equation*}
\lim_{n \to \infty} \int_{\R} F(t) \mu_n(dt) = \int_{\R} F(t) \mu(dt) \qquad \mbox{for every } F\in C_c(\R).
\end{equation*}
\end{proposition}

Proposition \ref{def:vague_conv} and Definition \ref{def:ss_def} are connected. Let us set the atomic measure $\delta_x : \mathcal{B} \to \{0,1\}$ as 
$$
\delta_x (B):=\begin{cases}
1 & \mbox{if } x \in B,\\
0 & \mbox{if } x \notin B,
\end{cases}
$$
and define 
\begin{equation}\label{eq:discrete_measure}
\mu_\bfnn(\cdot):=\frac{1}{d_\bfnn}\sum_{\substack{k=-d^-_\bfnn\\k\neq 0}}^{d^+_\bfnn} \delta_{\lambda_k\left(X^{(\bfnn)}\right)}(\cdot).
\end{equation}
\begin{lemma}\label{lem:equivalent_definition_distribution}
Let $\{X^{(\bfnn)}\}_\bfnn$ be a matrix sequence and $\omega : D \subset \R^M \to \R$ be a Borel-measurable function such that $D$ is a Borel set with $0<m(D)<\infty$, and such that the Lebesgue integral $\int_D \omega(\bfyy) m(d\bfyy)$ exists, finite or infinite. Then $\{X^{(\bfnn)}\}_\bfnn \sim_\lambda \omega$ iff $\mu_\bfnn \xrightarrow{v} \mu_\omega$, where $\mu_\omega$ is the probability measure on $\R$ associated to $\phi_\omega$ defined in \eqref{eq:rearrangment2}, i.e., such that $\mu_\omega((-\infty,t])=\phi_\omega(t)$.
\end{lemma}
\begin{proof}
By hypothesis, it is immediate to verify that $\omega$ is a random variable on the probability space $\left(D, \mathcal{B}\cap D, \frac{1}{m(D)}m(\cdot)\right)$, and that $\phi_\omega$ is its distribution function. Therefore, by standard theory, for every $F \in C_c(\R)$ it holds that 
\begin{align*}
&\int_{\R} F(t)\mu_\omega(dt) =\frac{1}{m(D)}\int_{D}F(\omega(\bfyy))m(d\bfyy), \qquad\frac{1}{d_\bfnn}\sum_{\substack{k=-d^-_\bfnn\\k\neq 0}}^{d^+_\bfnn}F\left(\lambda_k\left(X^{(\bfnn)}\right)\right)=\int_{\R} F(t)\mu_\bfnn(dt).
\end{align*}
Then the thesis follows at once from Proposition \ref{prop:vague_conv} and Definition \ref{def:ss_def}. 
\end{proof}

\begin{definition}[Discrete eigenvalues counting function]
	Let $\{X^{(\bfnn)}\}_\bfnn$ be a matrix sequence such that $\{X^{(\bfnn)}\}_\bfnn \sim_\lambda \omega$ and define the (discrete) \emph{eigenvalues counting function} $N(X^{(\bfnn)},\cdot) : \R \to \{0,1,\ldots,d_\bfnn\}$,
	\begin{equation}\label{eq:discrete_eig_counting_function}
	N(X^{(\bfnn)},t) := \left|\left\{ k=-d^-_\bfnn,\ldots,d^+_\bfnn \, : \, \lambda_k \left(X^{(\bfnn)}\right) \leq t  \right\}  \right|.
	\end{equation}
	It holds that
	$$
\mu_\bfnn((-\infty,t])	= \frac{N(X^{(\bfnn)},t)}{d_\bfnn}.
	$$
\end{definition}

 We have the following asymptotic relations.
\begin{theorem}[Discrete Weyl's law]\label{thm:discrete_Weyl_law}
Let $\{X^{(\bfnn)}\}_\bfnn$ be a matrix sequence such that $\{X^{(\bfnn)}\}_\bfnn \sim_\lambda \omega$. Suppose that $m\left(\left\{\bfyy\,:\, \omega(\bfyy)=t\right\}\right)=0$ for every $t\in R_\omega$ (or equivalently, that $\phi_{\omega}$ is continuous). Then 
\begin{subequations} 
\begin{equation}\label{eq:weyl-SLLN}
 \left\{X^{(\bfnn)}\right\}_\bfnn \sim_\lambda \omega^*(x), \quad x\in (0,1);
\end{equation} 
\begin{equation}\label{eq:discrete_Weyl_law1}
\lim_{\bfnn \to \infty}\frac{N(X^{(\bfnn)},t)}{d_\bfnn} = \phi_\omega(t), \quad \forall\,t \in \R.
\end{equation}
\end{subequations}
Define now the index $\hat{k} \in \{1,\ldots, d_\bfnn\}$ as $\vartheta_{X^{(\bfnn)}}(k)$, where $\vartheta_{X^{(\bfnn)}} : \{- d^-_\bfnn\left(X^{(\bfnn)}\right), \ldots,  d^+_\bfnn\left(X^{(\bfnn)}\right)\}\to \{1,\ldots, d_\bfnn\}$ is such that
\begin{equation}\label{def:re-index}
\hat{k}=\vartheta_{X^{(\bfnn)}}(k):= \begin{cases}
k + d^-_\bfnn\left(X^{(\bfnn)}\right) +1 & \mbox{if } k <0,\\
k + d^-_\bfnn\left(X^{(\bfnn)}\right) & \mbox{if } k >0,
\end{cases}\qquad 
\end{equation}
and let $\hat{k}=\hat{k}(\bfnn)$ be such that $\hat{k}(\bfnn)/d_\bfnn \to x_0^- \in (0,1)$ as $\bfnn\to \infty$. Then 
\begin{equation}\label{eq:discrete_Weyl_law2_1}
\lim_{\bfnn \to \infty} \lambda_{\hat{k}(\bfnn)}\left(X^{(\bfnn)}\right)  \in \left(\inf R_\omega, \sup R_\omega\right), \quad \sup_{t \in R_\omega} \left\{ t\leq \lim_{\bfnn \to \infty} \lambda_{\hat{k}(\bfnn)}\left(X^{(\bfnn)}\right)  \right\} = \lim_{x\to x_0^-} \omega^*\left(x\right).
\end{equation}
In particular, if $\omega^*$ is (left) continuous in $x_0$, then
\begin{equation}\label{eq:discrete_Weyl_law2_2}
\lim_{\bfnn \to \infty} \lambda_{\hat{k}(\bfnn)}\left(X^{(\bfnn)}\right) = \lim_{\bfnn \to \infty} \omega^*\left(\frac{\hat{k}(\bfnn)}{d_\bfnn}\right)=\omega^*\left(x_0\right).
\end{equation}
Finally, if $\lambda_{\hat{k}(\bfnn)}\left(X^{(\bfnn)}\right) \geq \inf\left(R_\omega\right)$  ($\leq \sup(R_\omega)$) definitely, then Equation \eqref{eq:discrete_Weyl_law2_2} holds for $x_0=0$ ($x_0=1$) as well. 
\end{theorem}
\begin{proof}
Let us observe that the function $\phi_\omega$ defined in \eqref{eq:rearrangment2} is the distribution function of $\omega$, while $\omega^*$ is the distribution function of $\phi_{\omega}$: $\omega^*$ is monotone increasing, right continuous and the possibly non-empty set of its point of discontinuity (jumps) is at most countable, see \cite[Lemma 6.3.10]{Taylor97}. By the assumptions on $\omega$ it follows that $\phi_{\omega}$ is continuous, then by standard results in Probability Theory it holds that $\phi_\omega \circ \omega:=X$ is uniformly distributed on $(0,1)$, i.e. $X\sim U(0,1)$, which implies that $\omega^*(X)$ and $\omega$ have the same distribution. Therefore, 
$$
\frac{1}{m(D)}\iint_{D}  F\left(\omega(\bfyy)\right)m(d\bfyy)= \mathbb{E}\left(F(\omega)\right)=\mathbb{E}\left(F(\omega^*(X))\right)= \int_0^1 F\left(\omega^*(x)\right)m(dx)  \quad\forall F \in C_c(\mathbb R).
$$
The above identity plus \eqref{def_asym-bis-Matrix} give \eqref{eq:weyl-SLLN}. Define now
	
	$$
	\phi_\bfnn (t) := \frac{N(X^{(\bfnn)},t)}{d_\bfnn}.
	$$
Clearly, $\phi_\bfnn$ is the distribution function of the s.p.m. $\mu_\bfnn$ in \eqref{eq:discrete_measure}, that is $\mu_\bfnn((-\infty,t])=\phi_\bfnn(t)$ for every $t\in \R$. By Lemma \ref{lem:equivalent_definition_distribution} we have that $\mu_\bfnn \xrightarrow{v} \mu_\omega$ and since $\phi_\omega$ is continuous, then $\mu_\omega((a,b))=\mu_\omega([a,b])$ for every $a<b \in \R$ and  
$$
\lim_{\bfnn \to \infty}\frac{N(X^{(\bfnn)},t)}{d_\bfnn}=  \lim_{\bfnn \to \infty}\phi_\bfnn(t) = \lim_{\bfnn \to \infty} \mu_\bfnn((-\infty,t]) = \mu((-\infty,t])=\phi_\omega(t) \qquad \mbox{for every } t \in \R,
$$
which is exactly \eqref{eq:discrete_Weyl_law1}. 

Without loss of generality and for the sake of simplicity, let now $d^-_\bfnn\left(X^{(\bfnn)}\right)=0$, such that $\hat{k}=k$, and define $\lambda_{k(\bfnn)}:= \lambda_{k(\bfnn)}\left(X^{(\bfnn)}\right)$.  By Equation \eqref{eq:discrete_Weyl_law1} and since $\phi_{\omega}$ is continuous, by a well know theorem of P\'{o}lya it holds that $\phi_\bfnn \to \phi_{\omega}$ uniformly, and therefore we can argue that
\begin{equation}\label{eq:2}
\lim_{\bfnn \to \infty} \frac{k(\bfnn)}{d_\bfnn}= \lim_{\bfnn \to \infty}\frac{N\left(X^{(\bfnn)},\lambda_{k(\bfnn)}\right))}{d_\bfnn} = \lim_{\bfnn \to \infty} \phi_\bfnn\left(\lambda_{k(\bfnn)}\right) = \phi_{\omega}\left(\lim_{\bfnn \to \infty}\lambda_{k(\bfnn)}\right).
\end{equation}
Since $k(\bfnn)/d_\bfnn \to x_0^-$ and since $\omega^*$ is right continuous with an at most countable number of points of jumps, then
the right-hand side of \eqref{eq:2} is well-defined. Moreover, since $x_0 \in (0,1)$ then by \eqref{eq:rearrangment2} and the Definition \ref{def:essential_range}, it holds that $\lim_{\bfnn \to \infty}\lambda_{k(\bfnn)}=t_0 \in (\inf R_\omega, \sup R_\omega)$. Finally, by the above relation \eqref{eq:2} and by \eqref{eq:rearrangment}, we can conclude that
$$
\lim_{\bfnn \to \infty}\omega^*\left( \frac{k(\bfnn)}{d_\bfnn}\right)= \sup_{t \in R_\omega} \left\{ t\leq t_0 \right\}.
$$
Let us observe now that $x_0$ is a jump discontinuity point if and only if there exist $t_1<t_2 \in R_\omega$ such that $R_\omega \subseteq (\inf R_\omega, t_1] \cup [t_2,\sup R_\omega)$ and $\phi_{\omega}(t)=x_0$ if and only if $t\in [t_1,t_2]$. Therefore, if $\omega^*$ is continuous in $x_0$, then $t_0=t_1=t_2 \in R_\omega$ and we have \eqref{eq:discrete_Weyl_law2_2}.
\end{proof}

In some sense, the limit relation \eqref{ss_def2} can be viewed as the strong law for large numbers for specially chosen sequences of dependent complex/real-valued random variables $\{ \lambda_{k(\bfnn)}\left(X^{(\bfnn)}\right)\}_\bfnn$. See \cite{Szego84} for the link between the spectrum of Hermitian matrices and equally distributed sequences (in the sense of Weyl) as defined in \cite[Definition 7.1]{Kuipers-Niederreiter}, and \cite{LL18} as a recent survey about equidistributions from a probabilistic point of view.

\begin{corollary}\label{cor:weak_clustering}
In the same hypothesis of Theorem \ref{thm:discrete_Weyl_law}, it holds that
$$
\lim_{\bfnn \to \infty}\frac{\left| \left\{ k=-d^-_\bfnn,\ldots,d^+_\bfnn \, : \, \lambda_k \left(X^{(\bfnn)}\right) \notin R_\omega \right\} \right|}{d_\bfnn} = 0,
$$
that is, the number of possible outliers is $o(d_\bfnn)$. In particular,
$$
\lim_{\bfnn \to \infty}\frac{\left| \left\{ k=-d^-_\bfnn,\ldots,d^+_\bfnn \, : \,  \lambda_k \left(X^{(\bfnn)}\right)\leq t,\, \lambda_k \left(X^{(\bfnn)}\right) \in R_\omega \right\} \right|}{d_\bfnn} = \phi_\omega(t).
$$
\end{corollary}
\begin{proof}
It is immediate from \eqref{eq:discrete_Weyl_law1}. Let us observe that
\begin{equation*}
0\leq \lim_{\bfnn \to \infty} \frac{\left| \left\{k \, : \, \lambda_k \left(X^{(\bfnn)}\right)< \inf R_\omega \right\} \right|}{d_\bfnn} \leq\lim_{n \to \infty} \frac{N(X^{(\bfnn)},\inf R_\omega)}{d_\bfnn}= \phi_\omega\left(\inf R_\omega\right) =0.
\end{equation*}
Moreover, since
\begin{align*}
\frac{\left| \left\{ k \, : \, \lambda_k \left(X^{(\bfnn)}\right) \notin R_\omega \right\} \right|}{d_\bfnn} &= \frac{\left| \left\{ k \, : \, \lambda_k \left(X^{(\bfnn)}\right) \in \R \right\} \right|}{d_\bfnn} - \frac{\left| \left\{ k \, : \, \lambda_k \left(X^{(\bfnn)}\right) \in R_\omega \right\} \right|}{d_\bfnn}\\
&= 1 - \frac{\left| \left\{ k \, : \, \lambda_k \left(X^{(\bfnn)}\right) \in R_\omega \right\} \right|}{d_\bfnn}\\
&= 1 - \frac{N(X^{(\bfnn)},\sup R_\omega)}{d_\bfnn}+\frac{\left| \left\{ k\, : \, \lambda_k \left(X^{(\bfnn)}\right)< \inf R_\omega \right\} \right|}{d_\bfnn},
\end{align*}
then passing to the limit we get
\begin{align*}
\lim_{\bfnn \to \infty} \frac{\left| \left\{ k \, : \, \lambda_k \left(X^{(\bfnn)}\right) \notin R_\omega \right\} \right|}{d_\bfnn} &= 1- \phi_\omega\left(\sup R_\omega\right) =0.
\end{align*}
The second part of the thesis follows instead by \eqref{eq:discrete_Weyl_law1} and the easy fact that
$$
\left| \left\{ k \, : \,  \lambda_k \left(X^{(\bfnn)}\right)\leq t,\, \lambda_k \left(X^{(\bfnn)}\right) \in R_\omega \right\} \right| = \left| \left\{ k \, : \,  \lambda_k \left(X^{(\bfnn)}\right)\leq t \right\} \right| - \left| \left\{ k \, : \,  \lambda_k \left(X^{(\bfnn)}\right)\leq t,\, \lambda_k \left(X^{(\bfnn)}\right) \notin R_\omega \right\} \right|,
$$
where
$$
\left| \left\{ k \, : \,  \lambda_k \left(X^{(\bfnn)}\right)\leq t,\, \lambda_k \left(X^{(\bfnn)}\right) \notin R_\omega \right\} \right| \leq \left| \left\{ k \, : \,  \lambda_k \left(X^{(\bfnn)}\right) \notin R_\omega \right\} \right|= o(d_\bfnn).
$$
\end{proof}

\begin{corollary}\label{cor:differentiability}
In the same hypothesis of Theorem \ref{thm:discrete_Weyl_law}, assume moreover that $\omega^*$ is absolutely continuous. Let $g: (\inf R_\omega,\sup R_\omega) \to \R$ be a differentiable real function and, using the same notation as in \eqref{def:re-index}, let $\{\hat{k}(\bfnn)\}_\bfnn$ be a sequence of integers such that
\begin{enumerate}[(i)]
	\item $\frac{\hat{k}(\bfnn)}{d_\bfnn} \to x_0 \in [0,1]$;
	\item $\lambda_{k(\bfnn)+1}\left(X^{(\bfnn)}\right)> \lambda_{k(\bfnn)}\left(X^{(\bfnn)}\right) \in [\inf R_\omega,\sup R_\omega]$ definitely for $\bfnn \to \infty$.
\end{enumerate} 
Then
$$
\lim_{\bfnn\to \infty} d_\bfnn\left[g\left(\lambda_{k(\bfnn)+1}\left(X^{(\bfnn)}\right) \right) - g\left(\lambda_{k(\bfnn)}\left(X^{(\bfnn)}\right) \right)\right] = \lim_{x\to x_0} \left(g(\omega^*(x))\right)' \quad \mbox{a.e.}
$$
\end{corollary}
\begin{proof}
Since $\omega^*$ is absolutely continuous then it is differentiable almost everywhere. Let $x_0 \in [0,1]$ such that $(\omega^*)'_{|x=x_0}$ exists. Then it is a straightforward calculation from \eqref{eq:discrete_Weyl_law2_2}:
\begin{align*}
\lim_{\bfnn \to \infty} d_\bfnn\left[g\left(\lambda_{k(\bfnn)+1}\left(X^{(\bfnn)}\right) \right) - g\left(\lambda_{k(\bfnn)}\left(X^{(\bfnn)}\right) \right)\right] &= \lim_{\bfnn \to \infty}\frac{g\left(\lambda_{k(\bfnn)+1}\left(X^{(\bfnn)}\right) \right) - g\left(\lambda_{k(\bfnn)}\left(X^{(\bfnn)}\right) \right) }{\frac{1}{d_\bfnn}}\\
&= \lim_{\bfnn \to \infty}\frac{g\left(\omega^*\left(\frac{\hat{k}(\bfnn)}{d_\bfnn} +\frac{1}{d_\bfnn} \right)\right) - g\left(\omega^*\left(\frac{\hat{k}(\bfnn)}{d_\bfnn}\right)\right)}{\frac{1}{d_\bfnn}}\\
&=\lim_{\bfnn \to \infty}\frac{g\left(\omega^*\left(x_0 +\frac{1}{d_\bfnn} \right)\right) - g\left(\omega^*\left(x_0\right)\right)}{\frac{1}{d_\bfnn}}\\
&=\lim_{x\to x_0} \left(g(\omega^*(x))\right)'.
\end{align*}
\end{proof}

\begin{corollary}\label{cor:discrete_Weyl_law}
Let $\omega$ satisfies the hypothesis of Theorem \ref{thm:discrete_Weyl_law} and moreover let suppose that $R_\omega$ is bounded, and that $\omega^*$ is (left) continuous. Then, in presence of no outliers (definitely), the absolute error between a uniform sampling of $\omega^*$ and the eigenvalues of $X^{(\bfnn)}$ converges to zero, namely
\begin{equation*}
\mathcal{A}_\bfnn:= \max_{\substack{k=-d^-_\bfnn,\ldots,d^+_\bfnn\\k\neq 0}} \left\{\left|\lambda_{k}\left(X^{(\bfnn)}\right) - \omega^*\left(\frac{\hat{k}}{d_\bfnn+1}\right)\right|\right\} \to 0 \qquad \mbox{as } \bfnn\to \infty.
\end{equation*}
\end{corollary}
\begin{proof}
Let us observe that since there are no outliers, then $\min(R_\omega)\leq \lambda_{k}\left(X^{(\bfnn)}\right) \leq \max(R_\omega)$ definitely for every $k$. Without loss of generality and for the sake of simplicity, let $d^-_\bfnn=0$. Suppose the thesis is false. Then there exists a sequence $k=k(\bfnn)$ such that $|\lambda_{k(\bfnn)}\left(X^{(\bfnn)}\right) - \omega^*\left(k(\bfnn)/(d_\bfnn+1)\right)| >\epsilon$ for some $\epsilon>0$. On the other hand, by the boundedness of $k(\bfnn)/(d_\bfnn+1)$ and by Theorem \ref{thm:discrete_Weyl_law}, passing to a subsequence we can assume that $k(\bfnn)/(d_{\bfnn}+1)$ converges to a point $x_0$ in $[0,1]$ and that \eqref{eq:discrete_Weyl_law2_2} holds. If $\omega^*\left(k(\bfnn)/(d_\bfnn+1)\right)=0$ definitely, then $|\lambda_{k(\bfnn)}\left(X^{(\bfnn)}\right) - \omega^*\left(k(\bfnn)/(d_\bfnn+1)\right)|\to 0$, which is a contradiction. At the same time, if  $\omega^*\left(k(\bfnn)/(d_\bfnn+1)\right)$ is not definitely identical to zero, then by passing again to a subsequence we can assume that $\omega^*\left(k(\bfnn)/(d_\bfnn+1)\right)\neq 0$, and by the boundedness of $R_\omega$
$$
\left|\lambda_{k(\bfnn)}\left(X^{(\bfnn)}\right) - \omega^*\left(\frac{k(\bfnn)}{d_\bfnn+1}\right)\right|\leq c\left|\frac{\lambda_{k(\bfnn)}\left(X^{(\bfnn)}\right)}{ \omega^*\left(\frac{k(\bfnn)}{d_\bfnn+1}\right)} - 1\right| \to 0,
$$
which is again a contradiction. 
\end{proof}

\subsection{Local and maximum spectral relative errors}
Exploiting the results obtained in the preceding section, we can now prescribe a way to measure the maximum relative error between two sequences of eigenvalues. 

Given two sequences of matrices $\{X^{(\bfnn)}\}_\bfnn, \{Y^{(\bfnn)}\}_\bfnn$ of the same dimension $d_\bfnn$, let us use the index notation as in \eqref{def:re-index}, in order to be able to compare the eigenvalues of $X^{(\bfnn)}$ and $Y^{(\bfnn)}$ from the lowest to the highest, for every $\bfnn$. See Figure \ref{fig:new_index}.

\begin{figure}[H]
 \begin{center}
	\begin{tikzpicture}
\draw [ ] ( -6,1)--( 6,1 );
\node ( left_end1 ) at (-6 , 1) {  };
\node ( right_end1 ) at (-6 , 1) {  };
\filldraw[black, fill=white] (0,1) circle (6pt);
\node ( zero1 ) at (0 ,1) {$0$};
\node (lambdaX) at (7.2,1.2) {$\lambda(X^{(6)})$};

\node[align=center,above] (X1) at (-5,1)  {\small{$k=\shortminus 2$}\\\small{$\hat{k}=1$}};
\filldraw[black,] (-5,1) circle (4pt);

\node[align=center,above] (X2) at (-3.2,1) {\small{$k=\shortminus 1$}\\\small{$\hat{k}=2$}};
\filldraw[black,] (-3.2,1) circle (4pt);

\node[align=center,above] (X3) at (1,1) {\small{$k= 1$}\\\small{$\hat{k}=3$}};
\filldraw[black,] (1,1) circle (4pt);

\node[align=center,above] (X4) at (2.5,1) {\small{$k= 2$}\\\small{$\hat{k}=4$}};
\filldraw[black,] (2.5,1) circle (4pt);

\node[align=center,above] (X5) at (4.2,1) {\small{$k= 3$}\\\small{$\hat{k}=5$}};
\filldraw[black,] (4.2,1) circle (4pt);

\node[align=center,above] (X6) at (5.8,1) {\small{$k= 4$}\\\small{$\hat{k}=6$}};
\filldraw[black,] (5.8,1) circle (4pt);

\node[align=center,below] (Y1) at (-5.6,-1)  {\small{$\hat{k}=1$}\\\small{$k=\shortminus 4$}};
\filldraw[black,] (-5.6,-1) circle (4pt);

\node[align=center,below] (Y2) at (-4,-1)  {\small{$\hat{k}=2$}\\\small{$k=\shortminus 3$}};
\filldraw[black,] (-4,-1) circle (4pt);

\node[align=center,below] (Y3) at (-2.5,-1)  {\small{$\hat{k}=3$}\\\small{$k=\shortminus 2$}};
\filldraw[black,] (-2.5,-1) circle (4pt);

\node[align=center,below] (Y4) at (-0.7,-1)  {\small{$\hat{k}=4$}\\\small{$k=\shortminus 1$}};
\filldraw[black,] (-0.7,-1) circle (4pt);

\node[align=center,below] (Y5) at (1.5,-1)  {\small{$\hat{k}=5$}\\\small{$k= 1$}};
\filldraw[black,] (1.5,-1) circle (4pt);

\node[align=center,below] (Y6) at (4.3,-1)  {\small{$\hat{k}=6$}\\\small{$k= 2$}};
\filldraw[black,] (4.3,-1) circle (4pt);

\draw [ ] ( -6,-1)--( 6,-1 );
\node ( left_end2 ) at (-6 , -1) {  };
\node ( right_end2 ) at (-6 , -1) {  };
\filldraw[black, fill=white] (0,-1) circle (6pt);
\node ( zero2 ) at (0 ,-1) {$0$};
\node (lambdaY) at (7.2,-1.3) {$\lambda(Y^{(6)})$};

\draw [dashed] (0,-2)--(0,-1.2); 
\draw [dashed] (0,-0.82)--(0,0.85); 
\draw [dashed] (0,1.2)--(0,2); 

\draw [dashed] (-6,1)--(-7,1); 
\draw [dashed] (6,1)--(7,1); 
\draw [dashed] (6,-1)--(7,-1); 
\draw [dashed] (-6,-1)--(-7,-1);

\draw[black] (4.3,-1) -- (5.8,1);
\draw[black] (1.5,-1) -- (4.2,1);
\draw[black] (-0.7,-1) -- (2.5,1);
\draw[black] (-2.5,-1) -- (1,1);
\draw[black] (-4,-1)  -- (-3.2,1);
\draw[black] (-5.6,-1) -- (-5,1);

	\end{tikzpicture}
\end{center}
\caption{Example of the new index notation $\hat{k}$ used to compare the eigenvalues of two matrices $X^{(\bfnn)}$, $Y^{(\bfnn)}$. In this case we set $\bfnn=6$ and non-null eigenvalues. Observe that the new labeling is made necessary due to the fact that, in general, $X^{(\bfnn)}$ and $Y^{(\bfnn)}$ can have different indices of inertia. In this example, $d^-_6(X^{(6)})=2$ and $d^-_6(Y^{(6)})=4$.}\label{fig:new_index}
\end{figure}
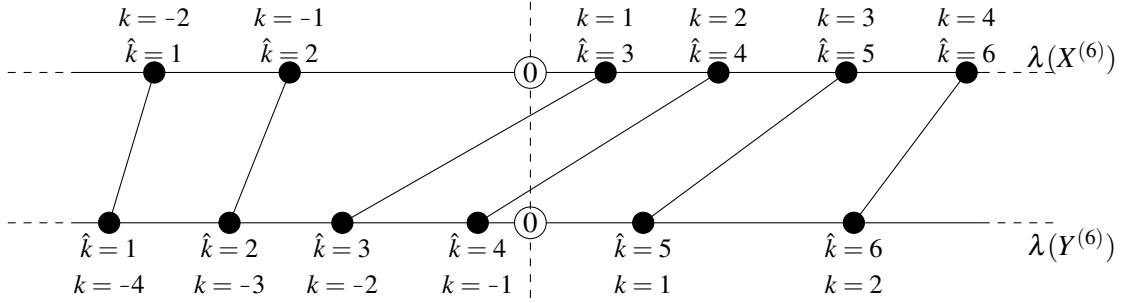

\begin{definition}[Local and maximum spectral relative errors]\label{def:max_spectral_err}
Given two sequences of matrices $\{X^{(\bfnn)}\}_\bfnn, \{Y^{(\bfnn)}\}_\bfnn$ of the same dimension $d_\bfnn$, define the following sequence $\{\mathcal{E}_\bfnn\}$: 
\begin{equation*}
\mathcal{E}_\bfnn:= \max_{\hat{k}=1,\ldots,d_\bfnn}\left\{\delta_{\hat{k}}^{(\bfnn)}\right\}, \qquad \delta_{\hat{k}}^{(\bfnn)}:=\begin{cases}\left| \frac{\lambda_{\hat{k}}\left(X^{(\bfnn)}\right)}{\lambda_{\hat{k}}\left(Y^{(\bfnn)}\right)}-1\right| & \mbox{if } \lambda_{\hat{k}}\left(Y^{(\bfnn)}\right) \neq 0,\\
0 & \mbox{if } \lambda_{\hat{k}}\left(Y^{(\bfnn)}\right) =\lambda_{\hat{k}}\left(X^{(\bfnn)}\right) = 0,\\
\infty & \mbox{if } \lambda_{\hat{k}}\left(X^{(\bfnn)}\right) \neq\lambda_{\hat{k}}\left(Y^{(\bfnn)}\right)=0.
\end{cases}
\end{equation*}
We call $\delta_{\hat{k}}^{(\bfnn)}$ \emph{the local spectral relative error}, and we call $\mathcal{E}:=\limsup_\bfnn \mathcal{E}_\bfnn$ the \emph{maximum spectral relative error}.
\end{definition}

\begin{theorem}\label{thm:necessary_cond_for_uniformity}
Fix two sequences of matrices of the same dimension $d_\bfnn$ such that $\{X^{(\bfnn)}\}_\bfnn \sim_\lambda \omega_1$ and  $\{Y^{(\bfnn)}\}_\bfnn \sim_{\lambda} \omega_2$. If both $\omega_1,\omega_2$ satisfy the hypothesis of Theorem \ref{thm:discrete_Weyl_law}, then
\begin{equation}\label{max_spectral_err_inequ}
\mathcal{E} \geq \max\left\{\sup_{\substack{x\in [0,1];\\x:\omega_2^*(x)\neq 0}} \left| \frac{\omega_1^*(x)}{\omega_2^*(x)} -1 \right|; \sigma\right\},
\end{equation}
where, recalling the definition of $\vartheta_{Y^{(\bfnn)}} : \{-d^-_\bfnn\left(Y^{(\bfnn)}\right) , \ldots,  d^+_\bfnn\left(Y^{(\bfnn)}\right)\} \to \{1,\ldots, d_\bfnn\}$ in \eqref{def:re-index}, 
\begin{equation*}
\sigma:=\begin{cases}
0 & \mbox{if } \omega_2^*(x)\neq 0 \, \forall \, x\in [0,1],\\
\sup_{\substack{k\in \mathbb{Z}\\k\neq 0}} \{\sigma_{k}\} & \mbox{if } \exists \, x_0\in [0,1] \mbox{ such that } \omega_2^*(x_0)=0,
\end{cases} \qquad  \sigma_{k}=\limsup_{\bfnn \to \infty}\left| \frac{\lambda_{\vartheta_{Y^{(\bfnn)}}(k)}\left(X^{(\bfnn)}\right)}{\lambda_{k}\left(Y^{(\bfnn)}\right)}-1\right| .
\end{equation*}
Moreover, if $\omega^*_1, \omega^*_2$ are (left) continuous and both the sequences do not have outliers, then equality holds in \eqref{max_spectral_err_inequ}.
\end{theorem}
\begin{remark}\label{rem_:0}
In the special case that $d^-_\bfnn\left(Y^{(\bfnn)}\right)=d^-_\bfnn\left(X^{(\bfnn)}\right)$, definitely, then we do not need the re-labeling and $\sigma_k = \limsup_{\bfnn \to \infty} \delta_{k}^{(\bfnn)}$, for any fixed $k \in \mathbb{Z}\setminus\{0\}$.
\end{remark}
\begin{proof}
Let us begin observing that 
$$
\mathcal{E}= \sup \left\{ \lim_{\bfnn_j \to \infty} \delta_{\hat{k}(\bfnn_j)}^{(\bfnn_j)} \, : \,   \left\{\delta_{\hat{k}(\bfnn_j)}^{(\bfnn_j)}\right\}_{\bfnn_j} \in F  \right\} \qquad F:= \left\{ \left\{ \delta_{\hat{k}(\bfnn_j)}^{(\bfnn_j)} \right\}_{\bfnn_j} \, : \, \exists \lim_{\bfnn_j \to \infty} \delta_{\hat{k}(\bfnn_j)}^{(\bfnn_j)} \; \mbox{finite or infinite} \right\}.
$$ 	
The thesis will follow immediately if we prove that equality in \eqref{max_spectral_err_inequ} holds  when $\omega^*_1, \omega^*_2$ are left continuous and $\lambda_{k}\left(X^{(\bfnn)}\right)\in R_{\omega_1}, \lambda_{k}\left(Y^{(\bfnn)}\right)\in R_{\omega_2}$ definitely. Indeed, the set of points of jumps for $\omega^*_1$ and $\omega^*_2$ is at most countable, and by passing to a subsequence, for every $x_0\in (0,1)$, $\omega^*_{2}(x_0)\neq 0$, there exists a sequence $\hat{k}(\bfnn(j))$ such that $\hat{k}(\bfnn(j))/d_{\bfnn_j}\to x_0^-$ and the limit $\lim_{\bfnn_j \to \infty} \delta_{\hat{k}(\bfnn_j)}^{(\bfnn_j)}$ exists finite, due to \eqref{eq:discrete_Weyl_law2_1}. Therefore, let us suppose that both $\omega_1, \omega_2$ are continuous and $\{X^{(\bfnn)}\}_\bfnn, \{Y^{(\bfnn)}\}_\bfnn$ do not have outliers. In this case, by Theorem \ref{thm:discrete_Weyl_law}, it happens that for every sequence $\delta_{\hat{k}(\bfnn_j)}^{(\bfnn_j)} \in F$, by passing to a subsequence,
$$
\lim_{\bfnn_j \to \infty} \frac{k(\bfnn_j)}{d_{\bfnn_j}} = x_0^- \in [0,1], \quad \lim_{\bfnn_j \to \infty} \lambda_{\hat{k}(\bfnn_j)}\left(X^{(\bfnn_j)}\right) = \omega_1^*(x_0) \quad \mbox{ and }\quad \lim_{\bfnn_j \to \infty} \lambda_{\hat{k}(\bfnn_j)}\left(Y^{(\bfnn_j)}\right) = \omega_2^*(x_0). 
$$
If $\omega^*_{2}(x)\neq 0$ for every $x\in [0,1]$, then we can conclude that 
$$
\mathcal{E} = \sup_{x\in [0,1]} \left| \frac{\omega_1^*(x)}{\omega_2^*(x)} -1 \right|.
$$
Notice that there can exists at most one point $x_0\in [0,1]$ such that $\omega_2^*(x_0)=0$. If $\omega_1^*(x_0)\neq 0$ then we can conclude that $\mathcal{E}=\infty$. Let us suppose now that $d_\bfnn=n$ and that $\omega^*_1(x),\omega^*_2(x)> 0$ in $(0,1]$, $\omega^*_1(0)=\omega_2^*(0)=0$ (such that, in particular, $d_n^-(X^{(n)})=d_n^-(Y^{(n)})=0$). If $k(n_j)/n_j\to x_0\in (0,1]$, then again it holds that 
$$
\lim_{n_j \to \infty} \delta_{k(n_j)}^{(n_j)} = \left|\frac{\omega_1^*(x_0)}{\omega_2^*(x_0)} -1\right|.
$$
We need then to study what happens to $\lim_{n_j \to \infty} \delta_{k(n_j)}^{(n_j)}$ when $k(n_j)/n_j\to 0$. If $\left\{k(n_j)\right\}$ is bounded, then
$$
\lim_{n_j \to \infty} \delta_{k(n_j)}^{(n_j)} \leq s:=\sup_{k \in \N}\left\{\limsup_{n_j \to \infty}\left| \frac{\lambda_{k}\left(X^{(n_j)}\right)}{\lambda_{k}\left(Y^{(n_j)}\right)}-1\right|\right\} .
$$
If $k(n_j)\to \infty$, then 
$$
\lim_{n_j \to \infty} \delta_{k(n_j)}^{(n_j)} = \lim_{n_j \to \infty} \left| \frac{\lambda_{k(n_j)}\left(X^{(n_j)}\right)}{\lambda_{k(n_j)}\left(Y^{(n_j)}\right)}-1\right| \leq \limsup_{k \to \infty}\limsup_{n_j \to \infty}\left| \frac{\lambda_{k}\left(X^{(n_j)}\right)}{\lambda_{k}\left(Y^{(n_j)}\right)}-1\right|\leq s.
$$
Therefore,
$$
\sup \left\{ \lim_{n_j \to \infty} \delta_{k(n_j)}^{(n_j)} \, : \,   \left\{\delta_{k(n_j)}^{(n_j)}\right\}_{n_j} \in F  \mbox{ and } k(n_j)/n_j \to 0 \right\} =s.
$$
In the general case, for $d_\bfnn^-\left(Y^{(\bfnn)}\right)>0$, the thesis follows by the same arguments after a suitable re-labeling of the indices. Indeed, observe that $\omega_2^*(x_0)=0$ if and only if $d^-_{\bfnn}\left(Y^{(\bfnn)}\right)/d_{\bfnn}\to x_0$ and that $\hat{k}(\bfnn_j)/d_{\bfnn_j}\to x_0 \in [0,1]$ if and only if $k(\bfnn_j)/d_{\bfnn_j}\to 0$, where $k(\bfnn_j)$ is the index of the eigenvalues of $Y^{(\bfnn)}$. Since
$$
 \delta_{k(\bfnn_j)}^{(\bfnn_j)} = \left| \frac{\lambda_{\hat{k}(\bfnn_j)}\left(X^{(\bfnn_j)}\right)}{\lambda_{\hat{k}(\bfnn_j)}\left(Y^{(\bfnn_j)}\right)}-1\right|=\left| \frac{\lambda_{\vartheta_{Y^{(\bfnn)}}(k(\bfnn_j))}\left(X^{(\bfnn_j)}\right)}{\lambda_{k(\bfnn_j)}\left(Y^{(\bfnn_j)}\right)}-1\right|,
$$ 
we can conclude. 
\end{proof}

\subsection{Linear self-adjoint differential operators and eigenvalue distribution}\label{ssec:application_to_self-adjoint_operators}
The asymptotic distribution of the eigenvalues for partial differential operators on general manifolds has been widely studied and developed, see for example \cite{Safarov,Levendorskii} and all the references therein. The topic is too vast to cover it properly, therefore we will concentrate our examples only to a couple of cases: Sturm-Liouville operators for the one dimensional case and elliptic self-adjoint operators for the multi-dimensional case, see Section \ref{sec:example} and Appendix \ref{sec:theory}. Nevertheless, the tools presented in this section can be applied to study the quality of a discretization scheme to preserve the discrete spectrum of many classes of self-adjoint operators. The approach is the following: given an operator $\mathcal{L}$ and its discretized version $\mathcal{L}^{(\bfnn)}$, if $\{h(d_\bfnn)\mathcal{L}^{(\bfnn)}\}_\bfnn \sim_{\lambda} \omega$ for a function $h : \N \to \R$ (which depends on the dimension of the underlying space and the higher order of derivatives involved), then study the asymptotic expression 
\begin{equation}\label{eq:continuous_weyl_function}
\zeta_\bfnn(t)=\frac{N^{(\bfnn)}\left(\mathcal{L},t\right)}{d_\bfnn}:=\frac{\left|k=-d^-_\bfnn,\ldots,d^+_\bfnn \, : \, h(d_\bfnn)\lambda_k\left(\mathcal{L}\right)\leq t\right|}{d_\bfnn}.
\end{equation}
We have the following result.
\begin{theorem}\label{thm:MSRE}
Let $\mathcal{L}$ be a self-adjoint linear operator and let  $\mathcal{L}^{(\bfnn)}$ be a $d_\bfnn\times d_\bfnn$ matrix obtained from $\mathcal{L}$ by a discretization scheme. Let $d^-_\bfnn, d^+_\bfnn$ be the negative and non-negative indices of inertia of $\mathcal{L}^{(\bfnn)}$. Suppose that:
\begin{enumerate}[(i)]
	\item $\lambda_{k}\left(\mathcal{L}^{(\bfnn)}\right) \to \lambda_{k}\left(\mathcal{L}\right)$ as $\bfnn \to \infty$ for every fixed $k \in \mathbb{Z}\setminus\{0\}$;\label{MSRE:item1}
	\item $\{h(d_\bfnn)\mathcal{L}^{(\bfnn)}\}_\bfnn \sim_{\lambda} \omega$ for some fixed $h : \N \to \R$;\label{MSRE:item2}
	\item $\omega$ satisfies the condition of Theorem \ref{thm:discrete_Weyl_law};\label{MSRE:item4}
	\item $\lim_{\bfnn \to \infty} \zeta_\bfnn(t) = \zeta(t)$, where $\zeta_\bfnn$ is defined in \eqref{eq:continuous_weyl_function} and $\zeta : \R \to [0,1]$ is continuous; \label{MSRE:item5}
\end{enumerate}
Then 
\begin{equation*}
\mathcal{E} \geq \sup_{\substack{x\in(0,1)\\x\,:\, \zeta^*(x)\neq 0}} \left|  \frac{\omega^*(x)}{\zeta^*(x)} -1\right|.
\end{equation*}
Moreover, if $\omega^*$,$\zeta^*$ are continuous and $h(d_\bfnn)\lambda_k\left(\mathcal{L}^{(\bfnn)} \right) \in R_\omega$ definitely, then equality holds.
\end{theorem}
\begin{proof}
	 Define 
	$$
	X^{(\bfnn)}:= h(d_\bfnn) \mathcal{L}^{(\bfnn)}, \qquad Y^{(\bfnn)}:= h(d_\bfnn)\diag_{\substack{k=-d^-_\bfnn,\ldots,d^+_\bfnn\\k\neq 0}} \left\{ \lambda_{k}\left(\mathcal{L}\right) \right\}.
	$$
By Item \eqref{MSRE:item5}, Lemma \ref{lem:equivalent_definition_distribution} and Theorem \ref{thm:discrete_Weyl_law}, it follows immediately that $\left\{Y^{(\bfnn)}\right\} \sim_\lambda \zeta^*$, that $\lambda_k\left(Y^{(\bfnn)} \right) \in R_{\zeta^*}= \overline{\zeta^*((0,1))}$ for every $k=-d^-_\bfnn,\ldots,d^+_\bfnn$, and that therefore the sequence $\left\{Y^{(\bfnn)}\right\}$ satisfies the hypothesis of Theorem \ref{thm:necessary_cond_for_uniformity}. By items \eqref{MSRE:item2}-\eqref{MSRE:item4}, we have that $\left\{X^{(\bfnn)}\right\}\sim_\lambda \omega$ and it satisfies too the hypothesis of Theorem \ref{thm:necessary_cond_for_uniformity}. Therefore, 
$$
\mathcal{E} \geq \max \left\{\sup_{\substack{x\in(0,1)\\x\,:\, \zeta^*(x)\neq 0}} \left|  \frac{\omega^*(x)}{\zeta^*(x)} -1\right|; \sigma \right\}.
$$ 
By Item  \eqref{MSRE:item1}, it follows that $\sigma=0$ and we conclude.
\end{proof}

The (weighted) Weyl function $\zeta$ is known for many kind of differential operators. See Subsection \ref{ssec:FD_nonuniform} and Subsection \ref{ssec:generalization} for a couple of examples.

\section{Numerical experiments}
All the computations are performed on MATLAB R2018b running on a desktop-pc with an Intel i5-4460H @3.20 GHz CPU and 8GB of RAM.
\subsection{Application to Euler-Cauchy differential operator}\label{sec:example}
We begin our analysis with respect to a toy-model example. In this subsection, the main focuses are:
\begin{itemize}
	\item to show numerical evidences of Theorem \ref{thm:necessary_cond_for_uniformity}, i.e., that the monotone rearrangement $\omega^*(x)$ of the spectral symbol $\omega(\bfyy)$ measures the maximum spectral relative error $\mathcal{E}$. See Subsections \ref{ssec:example_uniform_3_points}-\ref{ssec:FD_nonuniform}-\ref{ssec:galerkin};
	\item to disprove that, in general, a uniform sampling of the spectral symbol $\omega(\bfyy)$ can provide an accurate approximation of the eigenvalues of the weighted and un-weighted discrete operators $\hat{\mathcal{L}}^{(n)}$ and $\mathcal{L}^{(n)}$, respectively. See Subsection \ref{ssec:example_uniform_3_points};
	\item to show that a discretization scheme can lead to $\mathcal{E}=0$ if coupled with a suitable (possibly non-uniform) grid and an increasing order of approximation. See Subsections \ref{ssec:FD_nonuniform}-\ref{ssec:galerkin}.
\end{itemize}

Let us fix $\alpha>0$ and let us consider the following self-adjoint operator with Dirichlet BCs,
\begin{equation}\label{eq:Euler-Cauchy}
\Eulerc[u](x) := -\left(\alpha x^2 u'(x)\right)', \qquad \textnormal{dom}\left(\Eulerc\right)=W^{1,2}_0((1,\textrm{e}^{\sqrt{\alpha}})).
\end{equation} 
The formal equation is an Euler-Cauchy differential equation and by means of the Liouville transformation 
\begin{equation}\label{eq:liouville_transform}
y(x)= \int_1^x \left(\sqrt{\alpha}t\right)^{-1}dt,
\end{equation}
the operator \eqref{eq:Euler-Cauchy} is (spectrally) equivalent to 
\begin{equation}\label{eq:Euler-Cauchy_normal_form}
\Delta_{\textnormal{dir},\alpha}[v](y):= -v''(y) - \frac{\alpha}{4}v(y), \qquad \textnormal{dom}\left(\Delta_{\textnormal{dir},\alpha}\right)=W^{1,2}_0((0,1))),
\end{equation} 
which is a self-adjoint operator in Schroedinger form with constant potential $V(y)\equiv -\alpha/4$. For a general review, see for example \cite{Davies,Z15,EM99}. It is clear that
$$
\lambda_k \left(\Eulerc\right) = \lambda_k\left(\Delta_{\textnormal{dir},\alpha} \right) = \lambda_k\left(\Delta_{\textnormal{dir}} \right) +\frac{\alpha}{4} = k^2 \pi^2 +\frac{\alpha}{4}\quad \mbox{for every } k\geq 1,
$$    
where
\begin{equation*}
\Delta_{\textnormal{dir}}[v](y):= -v''(y), \qquad \textnormal{dom}\left(\Delta_{\textnormal{dir}}\right)=W^{1,2}_0((0,1))).
\end{equation*} 
For later reference, notice that
\begin{equation}\label{eq:shift}
\lim_{\alpha \to 0} \lambda_k \left(\Eulerc \right) = \lambda_k\left(\Delta_{\textnormal{dir}}  \right) = k^2 \pi^2 \quad \mbox{for every } k\geq 1,
\end{equation}
namely, the diffusion coefficient $p(x)=\alpha x^2$ produces a constant shift of $\alpha/4$ to the eigenvalues of the unperturbed Laplacian operator with Dirichlet BCs, i.e., $\Delta_{\textnormal{dir}}$.

We introduce the following definition.

\begin{definition}[Numerical and analytic spectral relative errors]\label{def:num_anal_error}
	Let $\Euler$ be the discrete differential operator obtained from \eqref{eq:Euler-Cauchy} by means of a generic numerical discretization method. If 
	$$
	\left\{(d_n+1)^{-2}\;\;\Euler\right\}_n \sim_{\lambda} \prescript{}{\alpha}{\omega}(x,\theta), \qquad (x,\theta) \in [1,\textrm{e}^{\sqrt{\alpha}}]\times [-\pi,\pi],
	$$
	then fix $n,n',r \in \N$, with $n'>>n$ and $r=r(n)\geq n$ and compute the following quantities
	$$
	\numerr = \left| \frac{\lambda_k\left(\Euler\right)}{\lambda_k\left(\mathcal{L}_{\textnormal{dir},\alpha x^2}^{(n')}\right)} -1\right|, \qquad 
	\analerr = \left| \frac{(d_n+1)^2\prescript{}{\alpha}{\omega^{*,(n)}_{r,k}}}{\lambda_k\left(\mathcal{L}_{\textnormal{dir},\alpha x^2}^{(n')}\right)} -1 \right|\qquad \mbox{for }k=1,\cdots, d_n.
	$$
	Specifically, $\prescript{}{\alpha}{\omega^{*,(n)}_{r,k}}= \prescript{}{\alpha}{\omega^*_{r}}\left(\frac{k}{d_n+1}\right)$, where $\prescript{}{\alpha}{\omega^*_{r}}$ is the (approximated) monotone rearrangement of the spectral symbol $\prescript{}{\alpha}{\omega}$ obtained by the procedure described in the algorithm of Section \ref{ss:rearrangment}. We call $\numerr$ the {\em numerical spectral relative error} and $\analerr$ the {\em analytic spectral relative error}. The difference between these definitions and \ref{def:max_spectral_err} is that in this case we are using as a comparing sequence the eigenvalues of the same discrete operator on a finer mesh, since supposedly we do not know the exact eigenvalues of the continuous operator but we do know that $\lambda_k\left(\mathcal{L}_{\textnormal{dir},\alpha x^2}^{(n')}\right)$ converges to the exact eigenvalue as $n'\to \infty$. We say that $\prescript{}{\alpha}{\omega}$ \emph{spectrally approximates} the discrete differential operator $\Euler$ if 
	$$
	\limsup_{n \to \infty}\; \left(\analerr\right) =0 \qquad \mbox{for every fixed } k.
	$$
\end{definition}

\subsubsection{Approximation by $3$-points central FD method on uniform grid}\label{ssec:example_uniform_3_points}

In our example, if we apply the standard central $3$-point FD scheme as in Appendix \ref{ssec:FD} with $\eta=1$ and $\tau(x)=x$, then the sequence of the weighted discretization matrices $\Eulerw= (n+1)^{-2}\Euler$ of the operator \eqref{eq:Euler-Cauchy} has spectral symbol 
$$
\prescript{}{\alpha}{\omega(x,\theta)}= \frac{\alpha x^2}{\left(\textrm{e}^{\sqrt{\alpha}}-1\right)^2}4\sin^2\left(\frac{\theta}{2}\right), \qquad \mbox{with}\quad D=[1,\textrm{e}^{\sqrt{\alpha}}]\times [0,\pi].
$$

Working with this toy-model problem in the $3$-points central FD scheme provides us a further advantage, since we can analytically calculate the monotone rearrangement $\prescript{}{\alpha}{\omega^*}$, or at least a finer approximation than $\prescript{}{\alpha}{\omega^*_r}$  which does not depend on the extra parameter $r$ and is less computationally expensive. Indeed, from equation \eqref{eq:rearrangment2} we have that
\begin{equation}\label{eq:phi_alpha}
\phi_{\prescript{}{\alpha}{\omega}} : \left[0,\frac{4\alpha\textrm{e}^{2\sqrt{\alpha}}}{\left(\textrm{e}^{\sqrt{\alpha}}-1\right)^2}\right] \to \left[0,1\right], \qquad 
\end{equation}
where
\begin{align*}
\phi_{\prescript{}{\alpha}{\omega}}(t)  &= \frac{1}{\pi\left(\textrm{e}^{\sqrt{\alpha}}-1\right)} m \left\{(x,\theta) \in [1,\textrm{e}^{\sqrt{\alpha}}]\times [0,\pi] \, : \, \frac{\alpha x^2\left(2-2\cos(\theta)\right)}{\left(\textrm{e}^{\sqrt{\alpha}}-1\right)^2} \leq t  \right\} \\
&= \frac{1}{\pi\left(\textrm{e}^{\sqrt{\alpha}}-1\right)}\cdot \begin{cases}
\int_1^{\textrm{e}^{\sqrt{\alpha}}} 2\arcsin\left( \frac{\left(\textrm{e}^{\sqrt{\alpha}}-1\right)\sqrt{t}}{2\sqrt{\alpha}x}\right) \, m(dx) & \mbox{if } t \in \left[0, \frac{4\alpha}{\left(\textrm{e}^{\sqrt{\alpha}}-1\right)^2}\right],\nonumber\\
\pi\left(\frac{\left(\textrm{e}^{\sqrt{\alpha}}-1\right)\sqrt{t}}{2\sqrt{\alpha}}-1\right) + \int_{\frac{\left(\textrm{e}^{\sqrt{\alpha}}-1\right)\sqrt{t}}{2\sqrt{\alpha}}}^{\textrm{e}^{\sqrt{\alpha}}} 2\arcsin\left( \frac{\left(\textrm{e}^{\sqrt{\alpha}}-1\right)\sqrt{t}}{2\sqrt{\alpha}x}\right) \, m(dx) & \mbox{if } t \in \left[\frac{4\alpha}{\left(\textrm{e}^{\sqrt{\alpha}}-1\right)^2}, \frac{4\alpha\textrm{e}^{2\sqrt{\alpha}}}{\left(\textrm{e}^{\sqrt{\alpha}}-1\right)^2}\right]
\end{cases}\nonumber\\
&= \frac{1}{\pi\left(\textrm{e}^{\sqrt{\alpha}}-1\right)}\cdot \begin{cases}
\Phi\left(t,\textrm{e}^{\sqrt{\alpha}}\right) - \Phi\left(t,1\right)& \mbox{if } t \in \left[0, \frac{4\alpha}{\left(\textrm{e}^{\sqrt{\alpha}}-1\right)^2}\right],\nonumber\\
\pi\left(\frac{\left(\textrm{e}^{\sqrt{\alpha}}-1\right)\sqrt{t}}{2\sqrt{\alpha}}-1\right) + 
\left[ \Phi\left(t,\textrm{e}^{\sqrt{\alpha}}\right) - \Phi\left(t,\frac{\left(\textrm{e}^{\sqrt{\alpha}}-1\right)\sqrt{t}}{2\sqrt{\alpha}}\right)  \right]& \mbox{if } t \in \left[\frac{4\alpha}{\left(\textrm{e}^{\sqrt{\alpha}}-1\right)^2}, \frac{4\alpha\textrm{e}^{2\sqrt{\alpha}}}{\left(\textrm{e}^{\sqrt{\alpha}}-1\right)^2}\right],
\end{cases}\nonumber
\end{align*}
and where
\begin{equation*}
\Phi(t,x) = \frac{\left(\textrm{e}^{\sqrt{\alpha}}-1\right)\sqrt{t}\log\left(2\alpha x\left(\sqrt{1-\frac{\left(\textrm{e}^{\sqrt{\alpha}}-1\right)^2}{4\alpha x^2}}\right) +1\right)}{\sqrt{\alpha}} + 2x\arcsin\left( \frac{\left(\textrm{e}^{\sqrt{\alpha}}-1\right)\sqrt{t}}{2\sqrt{\alpha}x}\right).
\end{equation*}

Since we have an analytic expression for $\phi_{\prescript{}{\alpha}{\omega}}(t)$, it is then possible to compute a numerical approximation of its generalized inverse $\prescript{}{\alpha}{\omega^*}$over the uniform grid $\left\{ \frac{k}{n+1} \right\}_{k=1}^n$, for example by means of a Newton method. This approximation of the monotone rearrangement does not depend on the extra parameter $r$: therefore, when we will compute the analytical spectral relative error with respect to $\prescript{}{\alpha}{\omega^*}$ we will write $\analerrExact$ without the subscript $r$.

\begin{figure}[H]
	\centering
	\includegraphics[width=12cm]{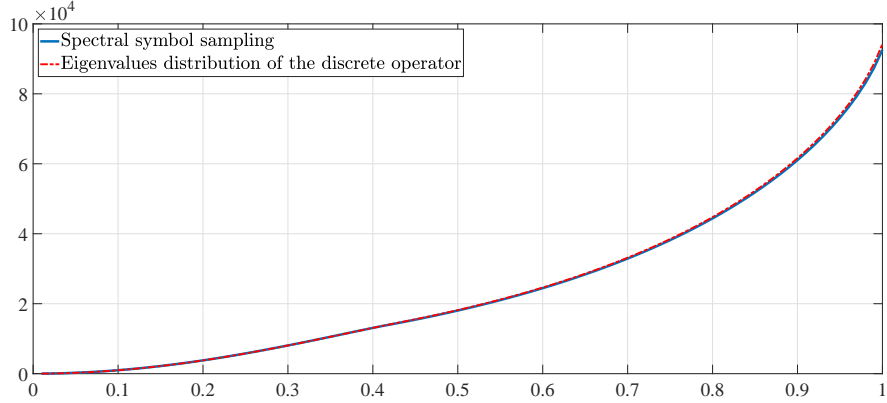}
	\captionof{figure}{For $\alpha=1$, comparison between the distribution of the first $n=10^2$ eigenvalues of the discrete operator $\Euler$ (red-dotted line) and the $n$-equispaced samples of $(n+1)^2\prescript{}{\alpha}{\omega^*_r}$ with $r=10^3$ (blue-continuous line). On the $x$-axis is reported the quotient $k/n$, for $k=1,\ldots,n$. The sovrapposition of the graphs is explained by Theorem \ref{thm:discrete_Weyl_law} and the limit \eqref{eq:discrete_Weyl_law2_2}.}\label{fig:eig_symbol_comparison}
	\end{figure}

In Figure \ref{fig:eig_symbol_comparison} it is possible to check that an equispaced sampling of $(n+1)^2\prescript{}{\alpha}{\omega^*_r}$ asymptotically distributes exactly as the eigenvalues of the unweighted discrete operator $\Euler$. Indeed, $\phi_{\prescript{}{\alpha}{\omega}}$ is continuous and strictly monotone increasing which implies that $\prescript{}{\alpha}{\omega}^*$ is continuous: then relation \eqref{eq:discrete_Weyl_law2_2} applies.  

Moreover, according with equation \eqref{eq:shift}, we observe that 
\begin{equation*}
\lim_{\alpha \to 0} \phi_{\prescript{}{\alpha}{\omega}}(t) = \frac{2}{\pi}\arcsin\left(\frac{\sqrt{t}}{2}\right) \qquad t\in [0,4], \qquad \lim_{\alpha \to 0}\prescript{}{\alpha}{\omega^*} (x) = 4\sin^2\left(\frac{\pi x}{2}\right) \qquad x \in [0,1],
\end{equation*}
which means that the  monotone rearrangement $\prescript{}{\alpha}{\omega^*}$ converges to the spectral symbol $\omega$ as $\alpha \to 0$, that is, to the spectral symbol which characterizes the differential operator $\Delta_{\textnormal{dir}}$ discretized by means of a 3-points FD scheme.  
The eigenvalues of $(n+1)^{-2}\left(\Delta_{\textnormal{dir}}^{(n)}\right)$ are the exact sampling of $\omega(\theta)=4\sin^2\left(\theta/2\right)$ over the uniform grid $\left\{ \frac{k\pi}{n+1} \right\}_{k=1}^n$, see \cite[p. 154]{Smith85}. This asymptotic behaviour reflects what we already observed in \eqref{eq:shift}.

All these remarks would suggest that $\prescript{}{\alpha}{\omega^*}$, or equivalently $\prescript{}{\alpha}{\omega^*_r}$, spectrally approximates the weighted discrete operator $\Eulerw$.  



Unfortunately, this conjecture looks to be partially proven wrong by Figure \ref{fig:analytical_err_1}: it shows the comparison between the graphs of the numerical spectral relative error $\numerr$ and the analytic spectral relative error $\analerr$, for several different increasing values of the parameter $r$. We observe a discrepancy in the analytical prediction of the eigenvalue error $\analerr$, for small $k<<n$, with respect to the numerical relative error $\numerr$.

	\begin{figure}
		\centering
	\includegraphics[width=12cm]{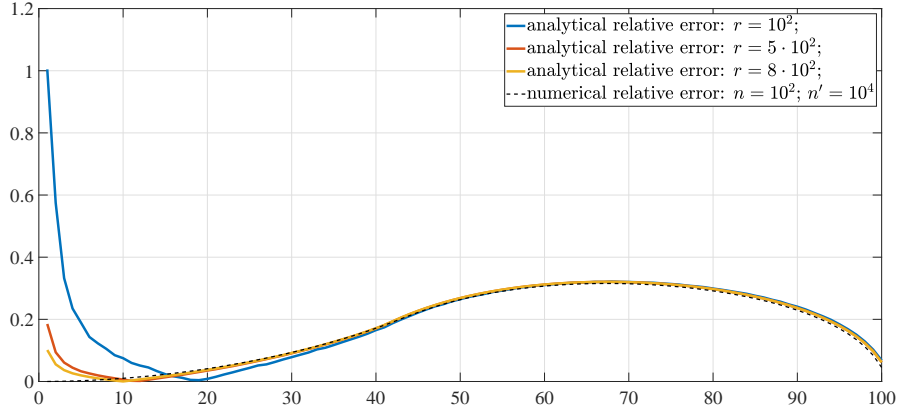}
	\captionof{figure}{Comparison between the numerical spectral relative error $\textbf{err}_k^{(n)}$ and the analytic spectral relative error $\tilde{\textbf{err}}_{k,r}^{(n)}$ for increasing $r=10^2,5\cdot 10^2,8\cdot10^2$. The values of $n$ and $n'$ are fixed at $10^2$ and $10^4$, respectively, and $\alpha=1$. The maximum discrepancy between the numerical relative error and the analytical relative errors is achieved for $k=1$ in all the three cases, and apparently it decreases as $r$ increases.}\label{fig:analytical_err_1}
		\end{figure}

In particular, the maximum discrepancy is achieved at $k=1$, for every $r$. The discrepancy apparently decreases as the number of grid points $r$ increases, as well observed in \cite[Figure 48]{GSERSH18} for some test-problems in the setting of Galerkin discretization by linear $C^0$ B-spline. In that same paper, some plausible hypothesis and suggestions were advanced:
\begin{itemize}
	\item the discrepancy could depend on the fact that it was used $\prescript{}{\alpha}{\omega^*_r}$ instead of $\prescript{}{\alpha}{\omega^*}$, and then that discrepancy should tend to zero in the limit $r\to \infty$, since $\prescript{}{\alpha}{\omega^*_r} \to \prescript{}{\alpha}{\omega^*}$.
	\item numerical instability of the analytic relative error $\analerr$ for small eigenvalues,\cite[Remark 3.1]{GSERSH18}.
	\item Change the sampling grid into an \textquotedblleft almost uniform\textquotedblright grid: see \cite[Rermark 3.2]{GSERSH18} for details.
\end{itemize}

The problem is that these hypothesis, which stem from numerical observations, cannot be validated: the descent to zero of the observed  discrepancy as $r$ increases is only apparent. Indeed, what happens is that, for every fixed $k$ it holds
\begin{equation}\label{eq:c_alpha,k}
\left|\analerr \right| \to c_{\alpha,k}= \frac{\frac{\alpha}{4}}{k^2\pi^2+\frac{\alpha}{4}}>0 \quad \mbox{as }r,n \to \infty,
\end{equation}
with $c_{\alpha,k}$ independent of $n$ and $n'$. This is the content of Proposition \ref{prop:non_good_approximation} and Remark \ref{rem:not_good_approximation}, with $p(x)=\alpha x^2,w(x)\equiv 1, q(x)\equiv 0$.

We then have a lower bound for the analytic spectral relative error which can not be avoided by refining the grid points. Of course, as $n\to \infty$, then $c_{\alpha,k} \to 0$ as $k$ increases. Those remarks are summarized in Table \ref{table:table_saturation}.


\begin{table}
	\centering
	\begin{tabular}{|c|l|c|c|c|} 
		\cline{3-5}
		\multicolumn{2}{l|}{}                                                                              & \multicolumn{3}{c|}{$|\tilde{\textbf{err}}^{(n)}_k/c_{\alpha,k} -1|$ }        \\ 
		\cline{3-5}
		\multicolumn{2}{c|}{}   & \multicolumn{1}{l|}{$n=10^2$ } & $n=10^3$             & $n=10^4$              \\ 
		\hline
		\multirow{3}{*}{\rotcell{$\alpha=0.1$ }}                                 & $c_{\alpha,1}=0.0025$                                                                                         & 0.0326                & 3.3223e-04  & 3.3283e-06   \\ 
		\cline{2-5}
		& $c_{\alpha,5}=1.0131e-04$                                                                                     & 20.3811               & 0.2076      & 0.0021       \\ 
		\cline{2-5}
		& $c_{\alpha,10}=2.5330e-05$                                                                                    & 325.3811              & 3.3222      & 0.0333       \\ 
		\hline\hline
		\multirow{3}{*}{\rotcell{$\alpha=1$ }}                                   & $c_{\alpha,1}=0.0247$                                                                                         & 0.0041               & 4.1363e-05  & 4.1438e-07   \\ 
		\cline{2-5}
		& $c_{\alpha,5}=0.0010$                                                                                         & 2.5395                & 0.0259      & 2.5899e-04   \\ 
		\cline{2-5}
		& $c_{\alpha,10}=2.5324e-04$                                                                                    & 40.6422               & 0.4136      & 0.0041       \\ 
		\hline\hline
		\multirow{3}{*}{\rotcell{$\alpha=2$}} & $c_{\alpha,1}=0.0482$                                                                                         & 0.0026                & 2.6120e-05  & 2.6167e-07   \\ 
		\cline{2-5}
		& $c_{\alpha,5}=0.0020$                                                                                         & 1.6056                & 0.0163      & 1.6354e-04   \\ 
		\cline{2-5}
		& $c_{\alpha,10}=5.0635e-04$                                                                                    & 25.7979               & 0.2612      & 0.0026       \\ 
		\hline\hline
		\multirow{3}{*}{\rotcell{$\alpha=5$}} & $c_{\alpha,1}=0.1124$                                                                                         & 0.0020              & 2.0008e-05  & 2.0044e-07   \\ 
		\cline{2-5}
		& $c_{\alpha,5}=0.0050$                                                                                         & 1.2389                &0.0125      & 1.2528e-04   \\ 
		\cline{2-5}
		& $c_{\alpha,10}=0.0013$                                                                                        & 20.4017               &0.2002      & 1.2528e-04   \\
		\hline
	\end{tabular}
\captionof{table}{For every fixed $k$ and $\alpha$, the analytic relative error $\analerrExact$ converges to the lower bound $c_{\alpha,k}$ as $n$ increases, where $c_{\alpha,k}$ is given in \eqref{eq:c_alpha,k}. Observe that $\analerrExact$ seems to be monotone decreasing of order $O(n^{-2})$. The approximation of $\prescript{}{\alpha}{\omega^*}$ is obtained evaluating $\prescript{}{\alpha}{\phi^{-1}}$ from \eqref{eq:phi_alpha} by means of the \texttt{fzero()} function from \textsc{Matlab} r2018b.}\label{table:table_saturation}
\end{table}

The problem lies on the wrong informal interpretation given to the limit relation in Definition \ref{def:ss_def}, and suggested by Remark \ref{rem:ssymbol_sampling}. Indeed, the asymptotic equality \eqref{def_asym-bis-Matrix} tells us that 
\begin{equation}\label{convergence_to_ICDF}
\left(\frac{k(n)}{n}, \lambda_{k(n)}\left(\Eulerw\right) \right)  \to \left(x, \prescript{}{\alpha}{\omega^*}(x) \right) \qquad \mbox{as }n\to \infty
\end{equation}
for every $k(n)$ such that $k(n)/n\to x \in [0,1]$, see Theorem \ref{thm:discrete_Weyl_law}. Therefore, since $\prescript{}{\alpha}{\omega^*}\in C([0,1])$ and $R_\omega$ is bounded, it follows that 
$$
\left\|\prescript{}{\alpha}{\omega^{*,(n)}_{r,k}}-\lambda_k\left(\Eulerw\right)\right\|_\infty\to 0 \quad \mbox{as } n\to \infty,
$$
by Corollary \ref{cor:discrete_Weyl_law} and as observed for example in \cite[Example 10.2 p. 198]{GS17}. On the contrary, a uniform sampling of the symbol $\omega$ does not necessarily provide an accurate approximation of the eigenvalues of the operator $\Eulerw$, in the sense of the relative error. The uniform sampling of the symbol works perfectly only for specific subclass of discretization schemes and operators, but it fails in general.

As a last remark, there does not exist an \textquotedblleft almost\textquotedblright uniform grid as well, nor in an asymptotic sense as described in \cite[Rermark 3.2]{GSERSH18}. 
Knowing the exact sampling grid which guarantees $\prescript{}{\alpha}{\omega^*}$ to spectrally approximate the discrete differential operator is equivalent to know the eigenvalue distribution of the original differential operator.

What we can do instead is to apply Theorem \ref{thm:MSRE}: by Item \eqref{item_spectral_conv_thm:FD_symbol} of Theorem \ref{thm:FD_symbol},  by the continuity of $\phi_{\prescript{}{\alpha}{\omega}}$ and by Equation \eqref{eq:3}, then items \eqref{MSRE:item1}-\eqref{MSRE:item5} of Theorem \ref{thm:MSRE} are satisfied. Moreover, since it holds that $\mathcal{L}_{\textnormal{dir},\alpha \tau(x)^2}^{(n)}$ does not have outliers by \cite[Theorem 2.2]{Serra00}, we conclude that
\begin{equation}\label{eq:limit_relative_error_FD_3_point}
\mathcal{E}= \max_{x\in[0,1]} \left| \frac{\prescript{}{\alpha}{\omega^*}(x)}{x^2\pi^2} -1\right|>0.
\end{equation}
In Figure \ref{fig:eig_distribution_VS_exact} and Table \ref{table:maximum_rel_error} it is numerically checked the validity of \eqref{eq:limit_relative_error_FD_3_point}.

\begin{table}
	\centering
	\begin{tabular}{|l|c|c|c|c|} 
		\cline{3-5}
		\multicolumn{1}{c}{}           &                                                                                                        & $n=10^2$              & $n=10^3$              & $n=5\cdot10^3$         \\ 
		\hline
		\multirow{2}{*}{$\alpha=0.5$ } & $|\frac{\max|\lambda^{(n)}_k/\lambda_k|}{\max| \prescript{}{\alpha}{\omega^*}(x)/x^2\pi^2|}-1|$  & 0.0104                & 0.0010                & 2.0853e-04             \\ 
		\cline{2-5}
		& $\bar{k}/n$                                                                                            & 0.7900                & 0.7880                & 0.7878                 \\ 
		\hline\hline
		\multirow{2}{*}{$\alpha=1$ }   & $|\frac{\max|\lambda^{(n)}_k/\lambda_k|}{\max| \prescript{}{\alpha}{\omega^*}(x)/x^2\pi^2|}-1|$  & 0.0158                & 0.0016                & 3.1754e-04             \\ 
		\cline{2-5}
		& $\bar{k}/n$                                                                                            & 0.6700                & 0.6680                & 0.6676                 \\ 
		\hline\hline
		\multirow{2}{*}{$\alpha=1.2$ } & $|\frac{\max|\lambda^{(n)}_k/\lambda_k|}{\max| \prescript{}{\alpha}{\omega^*}(x)/x^2\pi^2|}-1|$  & 0.0180                & 0.0018                & 3.6226e-04             \\ 
		\cline{2-5}
		& $\bar{k}/n$                                                                                            & 0.64                  & 0.6310                & 0.6302                 \\ 
		\hline\hline
		\multirow{2}{*}{$\alpha=3$ }   & $|\frac{\max|\lambda^{(n)}_k/\lambda_k|}{\max| \prescript{}{\alpha}{\omega^*}(x)/x^2\pi^2|}-1|$  & 0.0518                & 0.0097               &0.0032             \\ 
		\cline{2-5}
		& $\bar{k}/n$                                                                                            & 1               & 1              & 1                 \\ 
		\hline
	\end{tabular}\caption{In this table we check numerically the validity of Theorem \ref{thm:necessary_cond_for_uniformity} for different values of $\alpha$ and $n$. It can be seen that for every $\alpha$, as $n$ increases then the relative error between $\max_{k=1,\ldots,n}\left|\lambda_k\left(\Euler\right)/\lambda_k\left(\Eulerc\right)\right|$ and $\max_{x\in[0,1]}\left|\prescript{}{\alpha}{\omega^*}(x)/x^2\pi^2\right|$ decreases, validating \eqref{eq:limit_relative_error_FD_3_point}. In the table it is reported as well the ratio $\bar{k}/n$, where $\bar{k}$ is the $k$-th eigenvalue which achieves the maximum relative error between $\lambda_k\left(\Euler\right)$ and $\lambda_k\left(\Eulerc\right)$. We can notice that $\bar{k}/n$ tends to a fixed value in $(0,1]$ as $n$ increases. The approximation of $\prescript{}{\alpha}{\omega^*}$ is obtained evaluating $\prescript{}{\alpha}{\phi^{-1}}$ from \eqref{eq:phi_alpha} by means of the \texttt{fzero()} function from \textsc{Matlab} r2018b.}\label{table:maximum_rel_error}
\end{table}

\begin{figure}
	\centering
	\includegraphics[width=12cm]{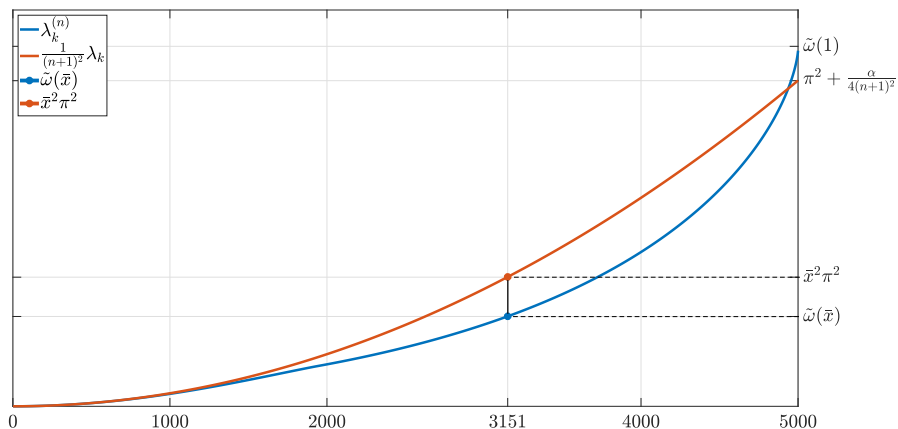}
	\captionof{figure}{For $\alpha=1.2$ and $n=5\cdot10^3$, comparison between the eigenvalues distribution of the weighted discrete differential operator $\Eulerw$ and the exact eigenvalues of the differential operator $\Eulerc$, weighted by $(n+1)^2$. The maximum spectral relative error $\mathcal{E}$ is obtained for $\bar{k}\approx 3151$, which corresponds to the maximum of $|\prescript{}{\alpha}{\omega^*}(x)/x^2\pi^2 -1|$, achieved at $\bar{x}\approx 0.6301 \approx \bar{k}/n$. See Table \ref{table:maximum_rel_error}.}\label{fig:eig_distribution_VS_exact}
\end{figure}

\subsubsection{Discretization by $(2\eta+1)$-points central FD method on non-uniform grid}\label{ssec:FD_nonuniform}
Clearly, everything said in the preceding Subsection \ref{ssec:example_uniform_3_points} remains valid even if we increase the order of accuracy of the FD method, namely, the spectral symbol $\prescript{}{\alpha}{\omega_\eta}$ of equation \eqref{eq:FD_symbol} does not spectrally approximate the discrete differential operator $\mathcal{L}_{\textnormal{dir},\alpha x^2}^{(n,\eta)}$, in the sense of the relative error, for any $\eta\geq 1$. 


What is interesting instead is to change the sampling grid and to increase the order of accuracy $\eta$ of the FD discretization method, see Appendix \ref{ssec:FD}. Indeed, as it was observed in \eqref{eq:limit_relative_error_FD_3_point}, it is not possible to achieve $\mathcal{E}=0$ if $\prescript{}{\alpha}{\omega^*_\eta}(x) \neq x^2\pi^2$. From \eqref{eq:FD_symbol}, for every $\eta\geq 1$ it is easy to check that $\max_{(x,\theta)\in[a,b]\times [0,\pi]}\prescript{}{\alpha}{\omega_{\eta}}(x,\theta)= \prescript{}{\alpha}{\omega^*_\eta}(1)\neq \pi^2$, and so we do not have any improvement by just increasing the order of accuracy $\eta$. On the other hand, observe that if we fix a new sampling grid $\left\{\bar{x}_j\right\}_{j=1}^n = \left\{\tau(x_j)\right\}_{j=1}^n$, with $\tau : [1,\textrm{e}^{\sqrt{\alpha}}]\to [1,\textrm{e}^{\sqrt{\alpha}}]$ a diffeomorphism, then
\begin{itemize}
	\item from Corollary \ref{cor:FD_uniform} and equation \eqref{eq:FD_symbol}, 
	$$
	\lim_{\eta\to \infty}\prescript{}{\alpha}{\omega_\eta}(x,\theta) = \frac{\alpha \tau(x)^2}{\left(\tau'(x)\right)^2\left(\textrm{e}^{\sqrt{\alpha}}-1\right)^2} \theta^2 \qquad \mbox{for every } (x,\theta) \in [1,\textrm{e}^{\sqrt{\alpha}}]\times [0,\pi];
	$$
	\item defining $\tau_2(y) = \textrm{e}^{\sqrt{\alpha}y}$, then $\frac{\alpha \tau_2(y)^2}{\left(\tau_2'(y)\right)^2}\equiv 1$.
\end{itemize}

In some sense, the spectral symbol $\prescript{}{\alpha}{\omega_\eta}$ suggests us to change the uniform grid $\left\{x_j\right\}_{j=1}^n \subset [1,\textrm{e}^{\sqrt{\alpha}}]$
by means of the diffeomorphism induced by the Liouville transformation. Indeed, from \eqref{eq:liouville_transform} we have that
\begin{equation*}
y(x)= \frac{\log(x)}{\sqrt{\alpha}} \quad \mbox{for } x \in [1,\textrm{e}^{\sqrt{\alpha}}], \qquad x(y)= \textrm{e}^{\sqrt{\alpha}y} \quad \mbox{for } y \in [0,1],
\end{equation*}  
and therefore we can compose a $C^\infty$-diffeomorphism $\tau: [1,\textrm{e}^{\sqrt{\alpha}}]\to [1,\textrm{e}^{\sqrt{\alpha}}]$ such that
\begin{equation}\label{eq:L-transform1}
\tau :  [1,\textrm{e}^{\sqrt{\alpha}}] \xrightarrow[]{\tau_1} [0,1] \xrightarrow[]{\tau_2}  [1,\textrm{e}^{\sqrt{\alpha}}], \qquad\tau_1(x) = \frac{1}{\textrm{e}^{\sqrt{\alpha}}-1}\left(x-1\right), \qquad \tau_2(y)= \textrm{e}^{\sqrt{\alpha}y}.
\end{equation}
The new non-uniform grid is then given by
\begin{equation}\label{eq:non_uniform_grid}
\left\{\bar{x}_j\right\}_{j=1}^n = \left\{\tau(x_j)\right\}_{j=1}^n,
\end{equation}
and it holds that
$$
\lim_{\eta\to \infty}\prescript{}{\alpha}{\omega_{\eta}^*}(x) =\pi^2 x^2 \qquad x \in [0,1].
$$
Moreover, with reference to \eqref{eq:continuous_weyl_function} and Theorem \ref{thm:MSRE}, it is easy to prove that
\begin{equation}\label{eq:3}
\zeta_n(t)= \frac{\left| \left\{ k=1,\ldots,n \, : \, \frac{\lambda_k\left(\Eulerc\right)}{n^2} \leq t \right\} \right|}{n}\to \zeta(t)=\frac{\sqrt{t}}{\pi}, \qquad \mbox{that is } \zeta^*(x)= \pi^2x^2.
\end{equation}
Since the discretization scheme is convergent, for every fixed $k$ the local spectral relative error $\delta_k^{(n)}$ (see Definition \ref{def:max_spectral_err}) converges to zero as $n \to \infty$, and by \cite[Theorem 2.2]{Serra00} it holds that $\lambda_k \left(\mathcal{L}_{\textnormal{dir},\alpha \tau(x)^2}^{(n,\eta)}\right) \in R_{\prescript{}{\alpha}{\omega_\eta}}$ for every $k,\ldots, n$, and for every fixed $\eta\geq 1$. From these remarks, we can apply Theorem \ref{thm:MSRE} and it follows that
$$
\mathcal{E}=\mathcal{E}(\eta) =  \left| \frac{\omega_{\eta}^*(x)}{\pi^2 x^2} -1 \right| \to 0 \qquad \mbox{as }\eta \to \infty.
$$

See Figure \ref{fig:comparison_non_uniform_FD} and Table \ref{tab:uniform_approx}. 


\begin{figure}
	\centering
	\subfloat[Uniform central FD with $\eta=1$]{
		\includegraphics[width=8cm]{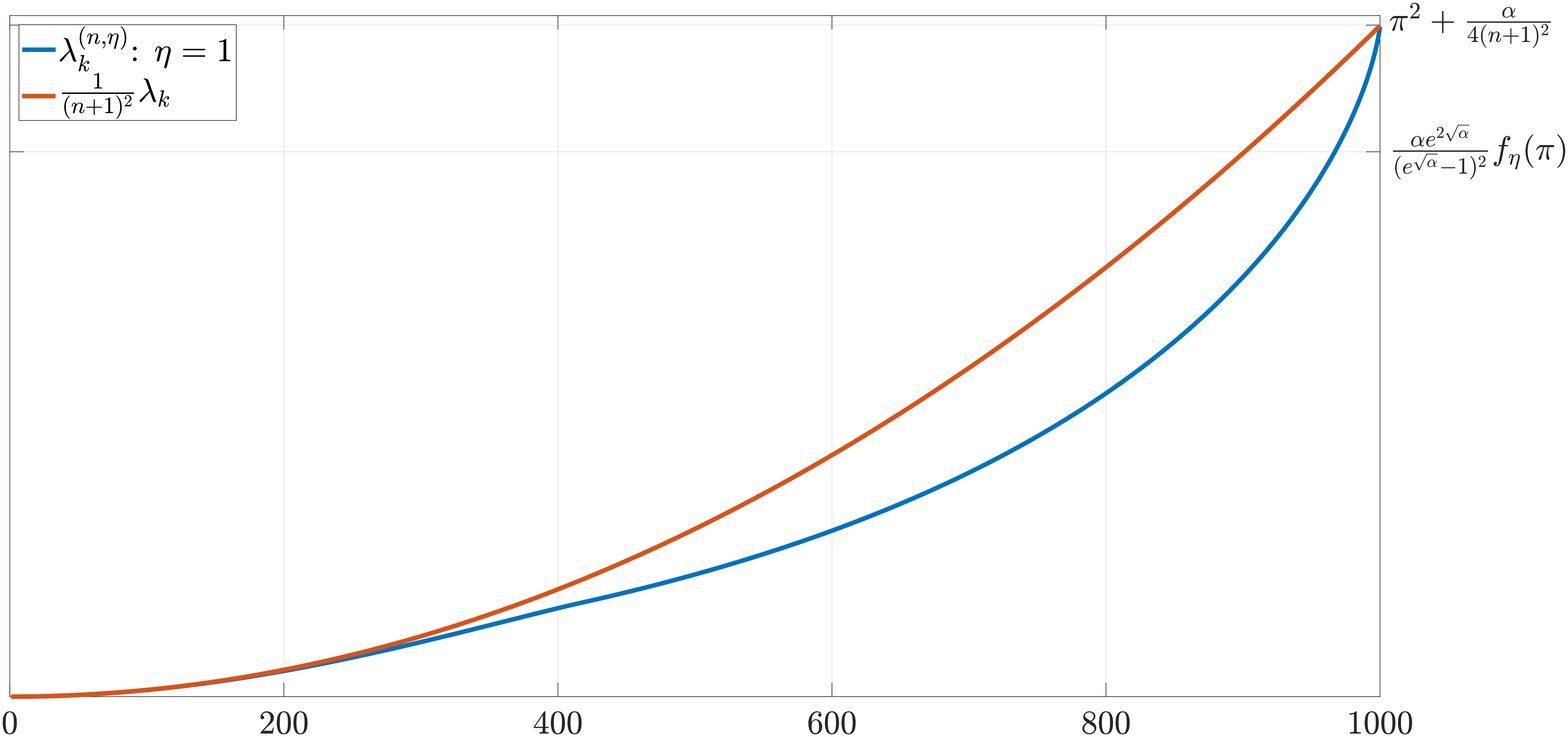}  
		\label{subfig:unif_FD_eta_1}
	}
	\subfloat[Uniform central FD with $\eta=15$]{
		\centering
		\includegraphics[width=8cm]{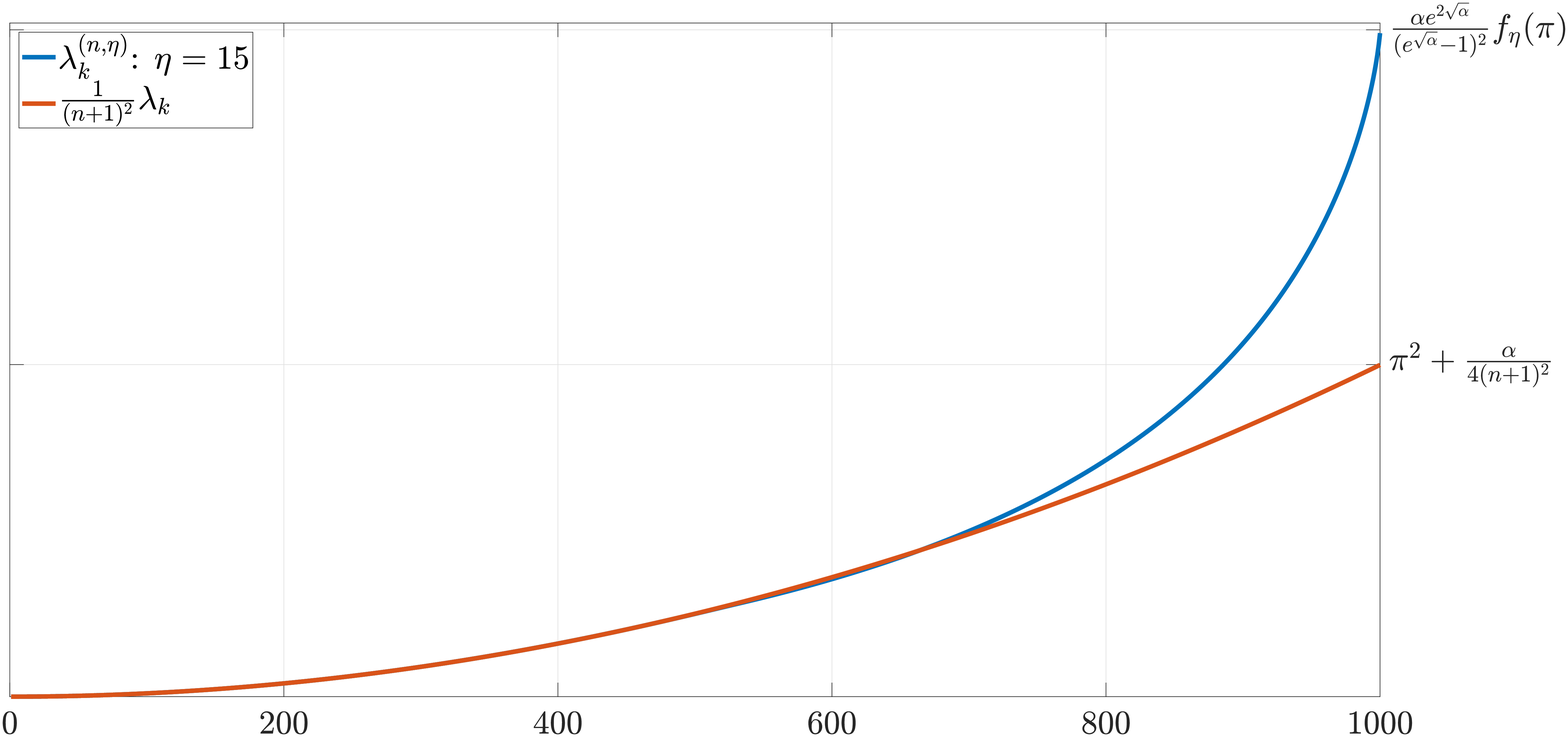}  
		\label{subfig:unif_FD_eta_15}
	}

\centering
\subfloat[Non-uniform central FD with $\eta=1$]{
	\includegraphics[width=8cm]{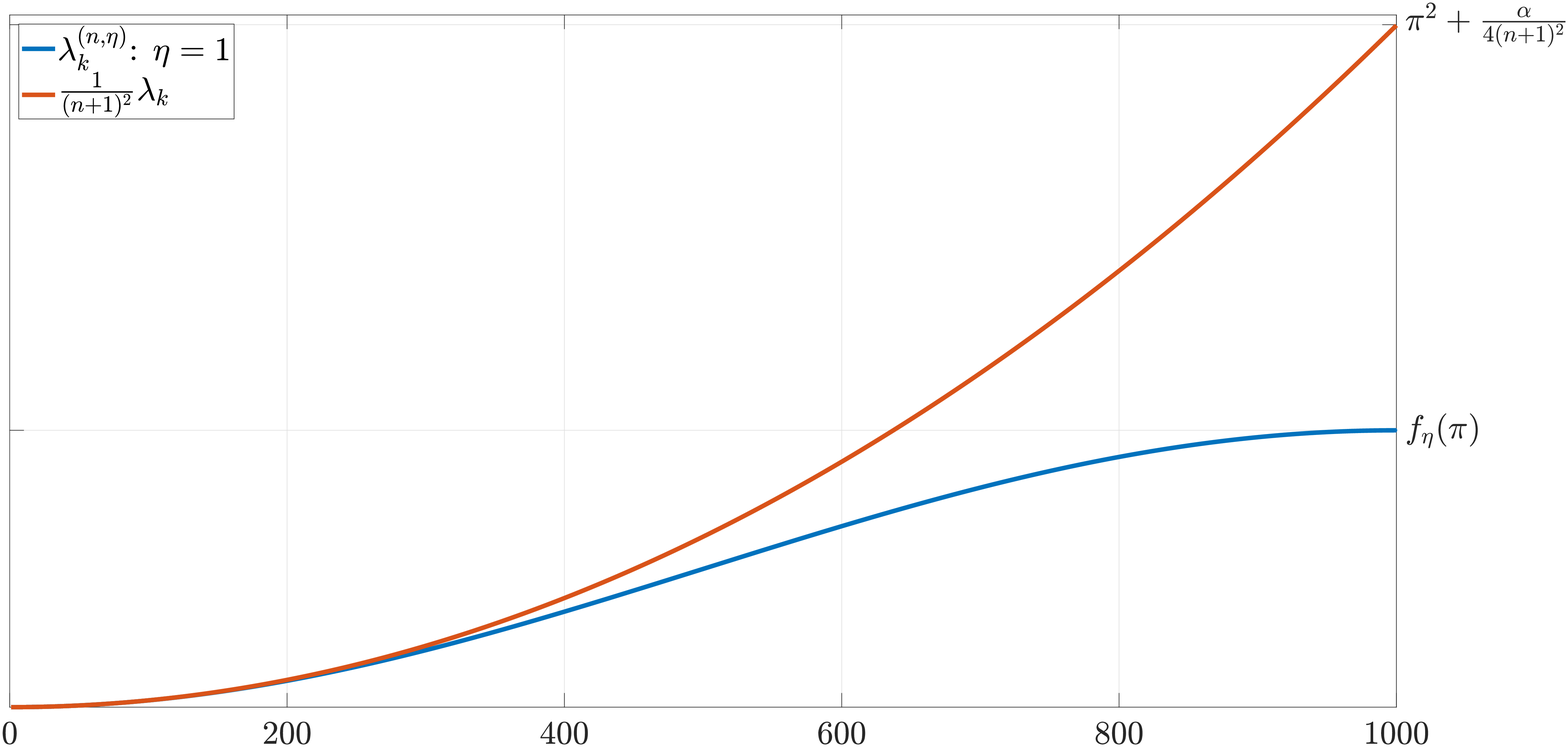}  
	\label{subfig:non_unif_FD_eta_1}
}
\subfloat[Non-uniform central FD with $\eta=15$]{
	\centering
	\includegraphics[width=8cm]{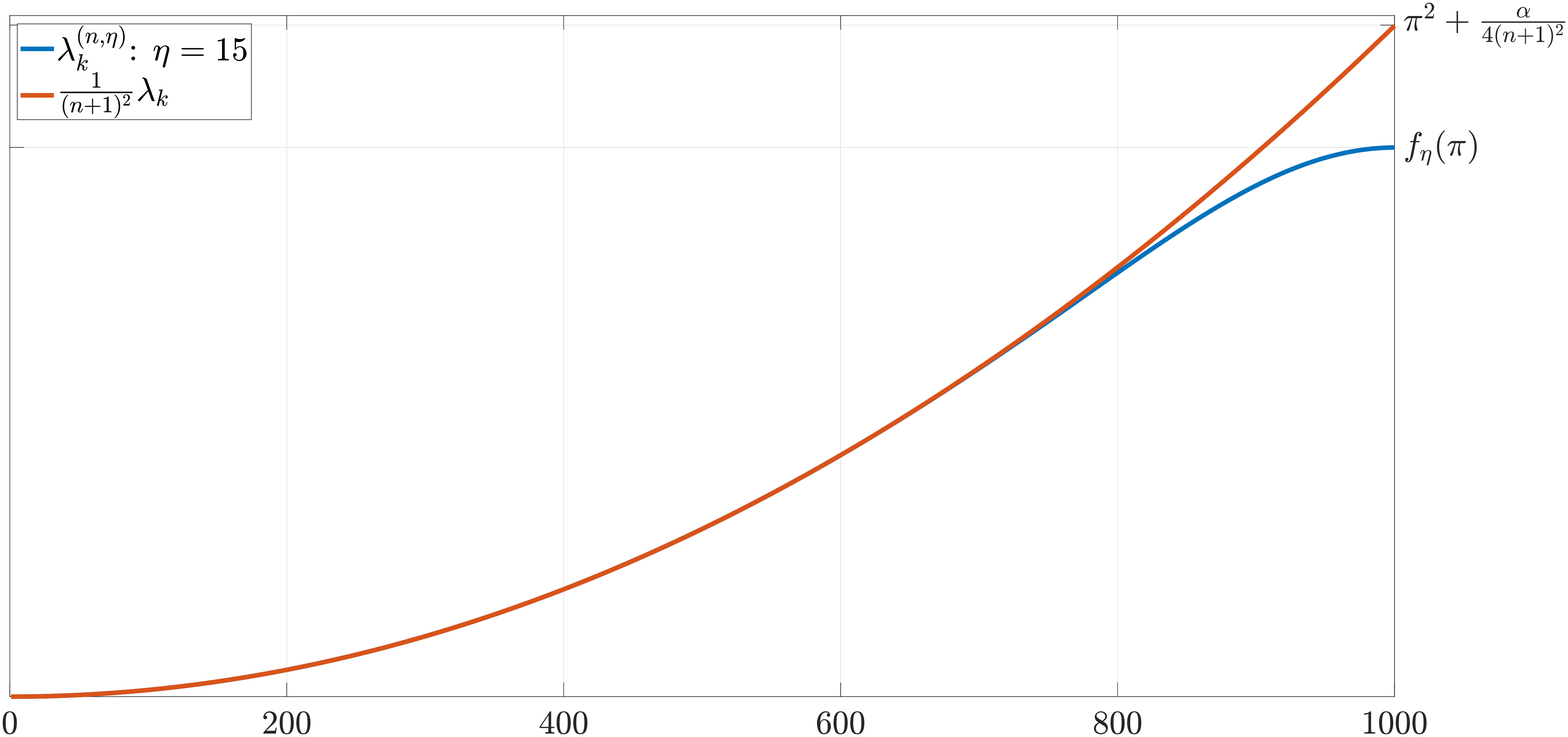}  
	\label{subfig:non_unif_FD_eta_15}
}
\captionof{figure}{Graphic comparison between the eigenvalues distribution of the weighted discrete differential operators $\hat{\mathcal{L}}_{\textnormal{dir},\alpha \tau(x)^2}^{(n,\eta)}$ obtained by means of $(2\eta+1)$-points central FD discretization on uniform ($\tau(x)=x$) and non-uniform grids ($\tau$ as in \eqref{eq:L-transform1}). The parameters $\alpha$ and $n$ are fixed, with $\alpha=1$ and $n=10^3$, while $\eta$ changes. Let us observe that in figures \ref{subfig:non_unif_FD_eta_1}, \ref{subfig:non_unif_FD_eta_15}, i.e., in the case of central FD discretization on the non-uniform grid given by \eqref{eq:non_uniform_grid}, the graph of the eigenvalue distribution seems to converge uniformly to the graph of the exact eigenvalues $(n+1)^{-2}\lambda_k$, as $\eta$ increases. The same phenomenon does not happen in the case of central FD discretization on uniform grid, as it is clear from figures \ref{subfig:unif_FD_eta_1},\ref{subfig:unif_FD_eta_15}. See Table \ref{tab:uniform_approx} for a numerical comparison of the maximum spectral relative errors.}\label{fig:comparison_non_uniform_FD}
\end{figure}

\begin{table}
	\centering
	\begin{tabular}{l|c|c|c|} 
		\cline{2-4}
		& \multicolumn{3}{c|}{$\mathcal{E}=\mathcal{E}(\eta)$ }  \\ 
		\cline{2-4}
		& \multicolumn{1}{c||}{$\eta=1$ } & \multicolumn{1}{c||}{$\eta=10$ } & $\eta=15$          \\ 
		\cline{2-4}
		& $n=10^2$                        & $n=10^3$                         & $n=15\cdot10^2$           \\ 
		\hline
		\multicolumn{1}{|l|}{uniform grid}     & 0.3155                          & 0.9057                          & 1.0101             \\ 
		\hline
		\multicolumn{1}{|l|}{non-uniform grid} & 0.5867
		                          & 0.2210                           & 0.1819             \\
		\hline
	\end{tabular}
	\captionof{table}{Comparison between the maximum spectral relative errors $\mathcal{E}$ of the discrete differential operators $\mathcal{L}_{\textnormal{dir},\alpha \tau(x)^2}^{(n,\eta)}$ obtained by means of $(2\eta+1)$-points central FD discretization on uniform ($\tau(x)=x$) and non-uniform grids ($\tau$ as in \eqref{eq:L-transform1}), for increasing values of $n$ and $\eta$. The parameter $\alpha$ is fixed, with $\alpha=1$. We observe that in the uniform grid case, as $n$ and $\eta$ increase the maximum increases as well. On the contrary, in the non-uniform grid case given by \eqref{eq:non_uniform_grid}, the maximum decreases as both $n$ and $\eta$ increase.}\label{tab:uniform_approx}
\end{table}

\subsubsection{IgA discretization by B-spline of degree $\eta$ and smoothness $C^{\eta-1}$}\label{ssec:galerkin}
In this subsection we continue our analysis in the IgA framework. We just collect all the numerical results of the tests, which confirm again what observed in subsections \ref{ssec:example_uniform_3_points} and \ref{ssec:FD_nonuniform}. The only difference relies on the fact that we took out the largest eigenvalues of the discrete operator $\mathcal{L}_{\textnormal{dir},\alpha \tau(x)^2}^{(n,\eta)} $. This is due to the fact that the IgA discretization suffers of a fixed number of outliers which depends on the degree $\eta$ and it is independent of $n$, see \cite[Chapter 5.1.2 p. 153]{CHB}. So, we consider only the eigenvalues $\lambda_k\left(\mathcal{L}_{\textnormal{dir},\alpha \tau(x)^2}^{(n,\eta)}\right)$ which belong to $R_{\prescript{}{\alpha}{\omega_{\eta}}}$.
We stress out the fact that the number of outliers is fixed for every $n$, in accordance with Corollary \ref{cor:weak_clustering}.


In Figure \ref{fig:comparison_non_uniform_IgA} we compare the graphs of the eigenvalue distributions between the discrete eigenvalues $\lambda_k\left(\mathcal{L}_{\textnormal{dir},\alpha \tau(x)^2}^{(n,\eta)}\right)$ and the exact eigenvalues $\lambda_k\left(\Eulerc\right)$, for different values of $\eta$ on uniform ($\tau(x)=x$) and non-uniform grids ($\tau$ as in \eqref{eq:L-transform1}). They line up with the numerics of Table \ref{table:maximum_rel_error_IgA_eta}: if the sampling grid is given by \eqref{eq:non_uniform_grid}, the maximum spectral relative error $\mathcal{E}$ decreases as the order of approximation increases.

	\begin{table}
		\centering
		\begin{tabular}{l|c|c|c|} 
			\cline{2-4}
			& \multicolumn{3}{c|}{$\mathcal{E}_n$ }  \\ 
			\cline{2-4}
			& \multicolumn{1}{c||}{$\eta=1$ } & \multicolumn{1}{c||}{$\eta=5$ } & $\eta=10$           \\ 
			\cline{2-4}
			& $n=10^2$                        & $n=5\cdot10^2$                        & $n=10^3$            \\ 
			\hline
			\multicolumn{1}{|l|}{uniform grid}     & 1.7653                          & 1.1971                          & 1.2005             \\ 
			\hline
			\multicolumn{1}{|l|}{non-uniform grid} & 0.4433                          & 0.0513                          & 0.0268              \\
			\hline
		\end{tabular}
	\captionof{table}{Comparison between the maximum spectral relative errors of the discrete differential operators $\mathcal{L}_{\textnormal{dir},\alpha \tau(x)^2}^{(n,\eta)}$ obtained by means of IgA discretization of order $\eta$ on uniform ($\tau(x)=x$) and non-uniform grids ($\tau$ as in \eqref{eq:L-transform1}), for increasing values of $\eta$ and $n$. The parameter $\alpha$ is fixed, with $\alpha=1$. We observe that in the non-uniform grid case given by \eqref{eq:non_uniform_grid}, the maximum spectral relative error decreases as $n$ and $\eta$ increase. Let us notice that we did not take in consideration the outliers, see Definition \ref{def:outliers}.}\label{table:maximum_rel_error_IgA_eta}
\end{table}

\begin{figure}
	\centering
	\subfloat[Uniform IgA with $\eta=1$]{
		\includegraphics[width=8cm]{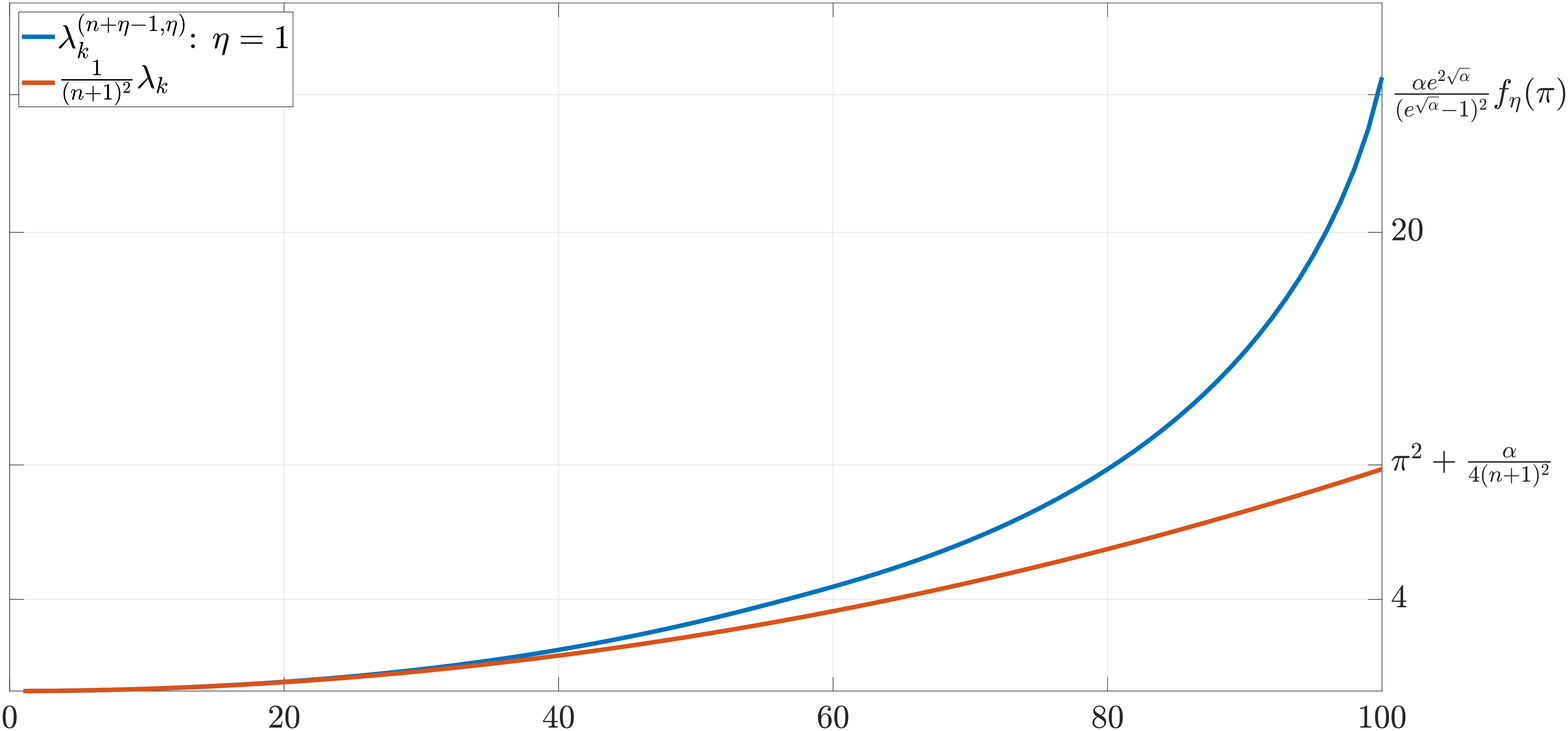}  
		\label{subfig:unif_IgA_eta_1}
	}
	\subfloat[Uniform IgA with $\eta=10$]{
		\centering
		\includegraphics[width=8cm]{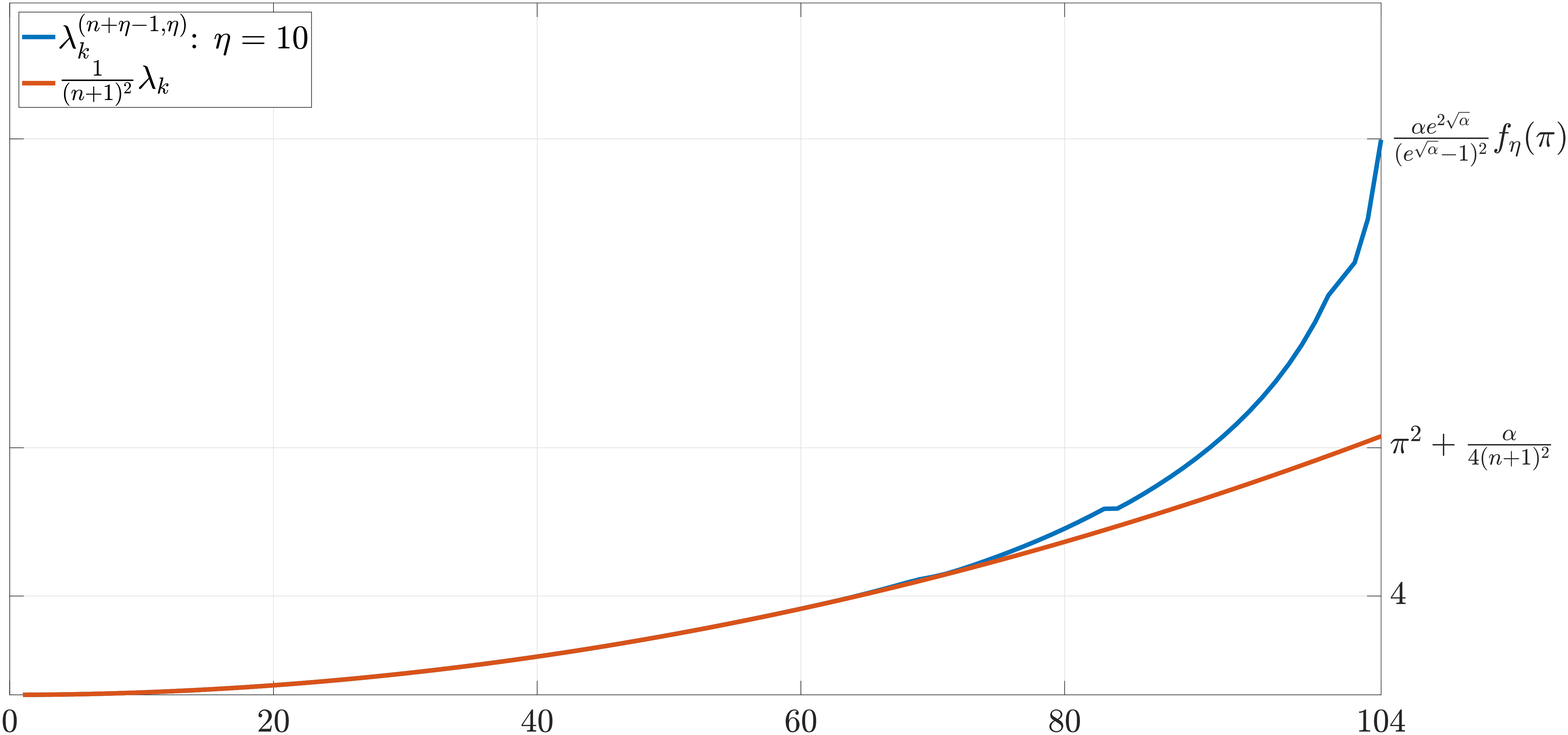}  
		\label{subfig:unif_IgA_eta_10}
	}
	
	\centering
	\subfloat[Non-uniform IgA with $\eta=1$]{
		\includegraphics[width=8cm]{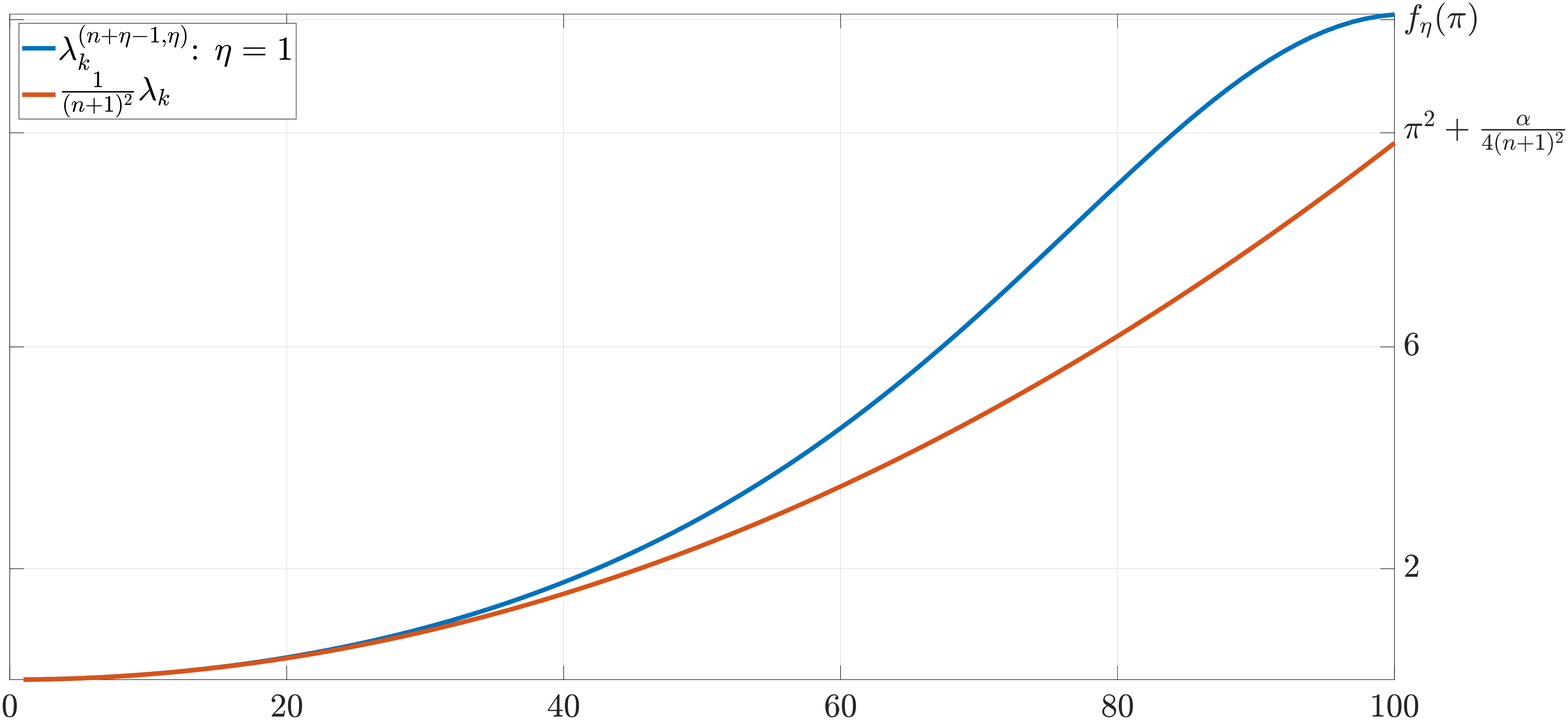}  
		\label{subfig:non_unif_IgA_eta_1}
	}
	\subfloat[Non-uniform IgA with $\eta=10$]{
		\centering
		\includegraphics[width=8cm]{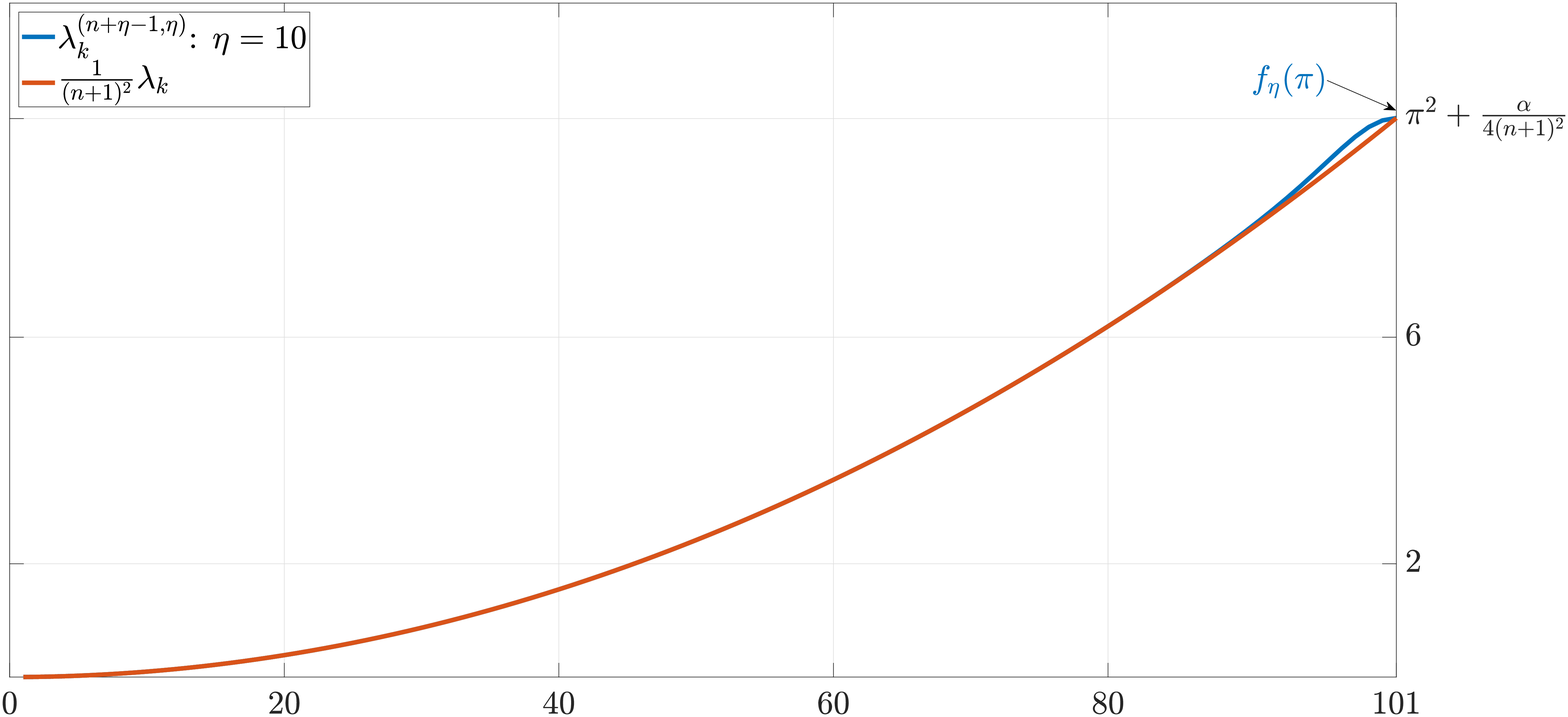}  
		\label{subfig:non_unif_IgA_eta_10}
	}
	\captionof{figure}{Graphic comparison between the eigenvalues distribution of the weighted discrete differential operators $\hat{\mathcal{L}}_{\textnormal{dir},\alpha \tau(x)^2}^{(n,\eta)}$ obtained by means of IgA discretization of order $\eta$ on uniform ($\tau(x)=x$) and non-uniform grids ($\tau$ as in \eqref{eq:L-transform1}). The parameters $\alpha$ and $n$ are fixed, with $\alpha=1$ and $n=10^2$, while $\eta$ changes. Let us observe that in figures \ref{subfig:non_unif_IgA_eta_1}, \ref{subfig:non_unif_IgA_eta_10}, i.e., in the case of IgA discretization on the non-uniform grid given by \eqref{eq:non_uniform_grid}, the graph of the eigenvalue distribution seems to converge uniformly to the graph of the exact eigenvalues $(n+1)^{-2}\lambda_k$, as $\eta$ increases. The same phenomenon does not happen in the case of IgA discretization on uniform grid, as it is clear from figures \ref{subfig:unif_IgA_eta_1},\ref{subfig:unif_IgA_eta_10}. Let us notice that we did not take in consideration the outliers, see Definition \ref{def:outliers}.}\label{fig:comparison_non_uniform_IgA}
\end{figure}

\subsection{$d$-dimensional Dirichlet Laplacian}\label{ssec:generalization}
The results of Section \ref{sec:asymptotic_distribution} can be generalized to study the maximum spectral relative error between more general linear self-adjoint differential operators and the numerical scheme implemented for their discretization. Let us sketch the ideas with a plain example.

Let $\Delta: W_0^{1,2}\left([0,1]^d\right) \to \textnormal{L}^2\left([0,1]^d\right)$ be the second order elliptic operator defined by the formal equation
$$
\Delta[u](\bfxx) := - \sum_{j=1}^d\frac{\partial^2 u}{\partial_{x^2_j}}(\bfxx), \qquad [0,1]^d \subset \R^d.
$$
It is well-known that it has empty essential spectrum and that 
\begin{align*}
&\lambda_k\left(\Delta\right) = \lambda_{k_1\cdots k_d}\left(\Delta\right) =\sum_{j=1}^d \pi^2k_j^2 \quad \mbox{for }k_1,\ldots,k_d\geq 1,\\
&\zeta(t) =\lim_{n \to \infty}\frac{\left|k=1,\ldots,n^d \, : \, \frac{\lambda_k\left(\Delta\right)}{(n+1)^2}\leq t\right|}{n^d} = c_d\left(\frac{t}{4\pi^2}\right)^{\frac{d}{2}}, \quad t \in [0,4\pi^2/(c_d)^{2/d}] ,
\end{align*}
where $c_d$ is the volume of the unit ball in $\R^d$, see \cite{Levendorskii}. Define
$$
Y^{(\bfnn)}:= (n+1)^{-2}\diag_{k=1,\ldots,d_\bfnn}\left\{\lambda_{k}\left(\Delta\right)\right\}, \qquad \bfnn=(\underbrace{n,\ldots,n}_{d-\mbox{times}}).
$$
Then it holds that $\left\{Y^{\bfnn}\right\}_\bfnn \sim_{\lambda} \zeta^*(x) =  4\pi^2\left(\frac{x}{c_d}\right)^{2/d}$, $x\in[0,1]$.

On the other hand, a discretization of the operator $\Delta$ by means of separation of variables and the classic equispaced $2d+1$-points FD scheme, leads to
$$
\Delta^{(\bfnn)}= (n+1)^2\sum_{k=1}^d \underbrace{I^{(n)}\otimes\cdots\otimes I^{(n)}}_{k-1}\otimes T^{(n)}\left(2-2\cos(\theta_k)\right)\otimes \underbrace{I^{(n)}\otimes\cdots\otimes I^{(n)}}_{d-k}, \qquad \theta_k \in [0,\pi],
$$
where $I^{(n)}$ is the identity matrix and $\otimes$ is the Kronecker product, and it holds 
$$
\left\{(n+1)^{-2}\Delta^{(\bfnn)}\right\}_\bfnn=\left\{\hat{\Delta}^{(\bfnn)}\right\}_\bfnn \sim_\lambda \omega\left(\bftheta\right)= \sum_{k=1}^d (2-2\cos(\theta_k)), \qquad \bftheta=(\theta_1,\ldots,\theta_d) \in [0,\pi]^d,
$$ 
see \cite[Chapter 7.3]{GS18}. Since the discretization scheme is convergent, for every fixed $k$ the local spectral relative error $\delta_k^{(\bfnn)}$ (see Definition \ref{def:max_spectral_err}) converges to zero as $\bfnn \to \infty$, and by \cite[Theorem 2.2]{Serra00} it holds that $\lambda_k \left(\Delta^{(\bfnn)}\right) \in R_{\omega}$ for every $k, \bfnn$. Moreover, it is immediate to check that both $\omega^*,\zeta^*$ are continuous, therefore we can apply Theorem \ref{thm:MSRE} to get the maximum spectral relative error: 
\begin{align*}
\mathcal{E}=\sup_{x\in(0,1]} \left| \frac{\omega^*(x)}{\zeta^*(x)}  -1\right|\geq \left| \frac{\omega^*(1)}{\zeta^*(1)}  -1\right|= \left|\frac{4d}{\frac{4\pi^2}{(c_d)^{2/d}}}  -1\right|>0.
\end{align*}

\section{Conclusions}\label{sec:conlcusions}

Given a differential operator $\mathcal{L}$ discretized by means of a numerical scheme, the knowledge of the spectral symbol $\omega$ provides a way to measure how far the discretization method is from a uniform approximation of all the eigenvalues of the original differential operator. Moreover, the symbol itself suggests a reparametrization of the uniform grid which, if coupled to an increasing refinement of the order of approximation of the method, it allows to obtain $\mathcal{E}=0$, see \eqref{max_rel_spectrum}. This is crucial in engineering applications. In light of this, it becomes a priority to devise new specific discretization schemes with
mesh-dependent order of approximation which guarantee a good balance between convergence to zero of the maximum relative spectral error and computational costs.

\appendix

\section{Auxiliary results}\label{sec:theory}

\begin{proposition}\label{prop:non_good_approximation}
	Let us consider a Sturm-Liouville operator defined by the formal equation
	\begin{align*}
\mathcal{L}_{\textnormal{BCs},p(x),q(x),w(x)}[u](x):= \frac{1}{w(x)}\left[-\partial_x\left(p(x)\partial_xu(x)\right) +q(x)u(x)\right], \qquad x\in (a,b)\subset \R, 
	\end{align*}
	such that 
	\begin{enumerate}[(i)]
		\item $p, p',w,w',q,(pw)',(pw)'' \in C([a,b])$;
		\item $p,w>0$;
		\item regular BCs.
	\end{enumerate}
	Discretize the above self-adjoint linear differential operator $\mathcal{L}_{\textnormal{BCs},p(x),q(x),w(x)}$ by means of a numerical scheme, and let $\mathcal{L}_{\textnormal{BCs},p(x),q(x),w(x)}^{(n,\eta)}$ be the correspondent discrete operator, where $n$ is the mesh finesse parameter and $\eta$ is the order of approximation of the numerical scheme. Define $B=\int_a^b \sqrt{\frac{w(x)}{p(x)}}m(dx)$. If:
	\begin{enumerate}[(a)]
		\item for every fixed $k\in \N$, 
		\begin{equation*}
		\lim_{n \to \infty} \lambda_k\left( \mathcal{L}_{\textnormal{BCs},p(x),q(x),w(x)}^{(n,\eta)}\right) = \lambda_k\left(\mathcal{L}_{\textnormal{BCs},p(x),q(x),w(x)}\right);
		\end{equation*}
		\item there exists $\omega : [a,b]\times [-\pi,\pi] \to \R$, $\omega \in \textnormal{L}^1\left([a,b]\times [-\pi,\pi]\right)$ such that
		\begin{equation*}
		\left\{(d_n+1)^{-2}\mathcal{L}_{\textnormal{BCs},p(x),q(x),w(x)}^{(n,\eta)}\right\}_n=\left\{\hat{\mathcal{L}}_{\textnormal{BCs},p(x),q(x),w(x)}^{(n,\eta)}\right\}_n \sim_{\lambda}\omega(x,\theta) \quad (x,\theta) \in [a,b]\times [-\pi,\pi];
		\end{equation*}
		\item\label{item:tilde_omega_asymptotics} for every $\eta$, the monotone rearrangement $\omega^*_\eta$, as defined in \eqref{eq:rearrangment}, is such that
		\begin{equation*}
		\omega^*_\eta(x) \sim \frac{x^2\pi^2}{B^2}, \qquad \mbox{as } x\to 0.
		\end{equation*} 
	\end{enumerate}
	Then, for every fixed $k\in \N$ 
	$$
		\lim_{n \to \infty} \left|\frac{(d_n+1)^2\omega^*_\eta\left(\frac{k}{d_n+1}\right)}{\lambda_k\left(\mathcal{L}_{\textnormal{BCs},p(x),q(x),w(x)}^{(n,\eta)} \right)} -1 \right| =   	\lim_{n \to \infty} \left|\frac{(d_n+1)^2\omega^*_\eta\left(\frac{k}{d_n+1}\right)}{\lambda_k\left(\prescript{B}{0}{\mathcal{L}_{\textnormal{BCs},V}}\right)} -1 \right| =c_k\geq 0, \qquad \lim_{k \to \infty} c_k =0,
	$$
	with $c_k$ independent of $\eta$, and 
	where $\prescript{B}{0}{\mathcal{L}_{\textnormal{BCs},V}}$ is the differential operator associated to the normal form of  $\prescript{b}{a}{\mathcal{L}_{\textnormal{BCs},p(x),q(x),w(x)}}$ by the Liouville transform $y(x):=\int_a^x\sqrt{\frac{w(t)}{p(t)}}m(dt)$. 
	
	In particular,
	\begin{equation*}
	\lim_{n \to \infty}\analerr = c_k, 
	\end{equation*}
	where $\analerr$ is the analytic spectral error of Definition \ref{def:num_anal_error}.
\end{proposition}  
\begin{proof}
	By standard theory it holds that 
	$$
\lambda_k\left(\mathcal{L}_{\textnormal{BCs},p(x),q(x),w(x)} \right)=	\lambda_k\left(\prescript{B}{0}{\mathcal{L}_{\textnormal{BCs},V(y)}}\right) \qquad \mbox{and} \qquad	\lambda_k\left(\prescript{B}{0}{\mathcal{L}_{\textnormal{dir},V\equiv 0}}\right) =  \lambda_k\left(\prescript{B}{0}{\Delta_{\textnormal{dir}}}\right)=\frac{k^2\pi^2}{B^2}.
	$$
	Therefore, from item \eqref{item:tilde_omega_asymptotics}
	\begin{equation}
	\lim_{n \to \infty} (d_n+1)^2\omega^*_\eta\left(\frac{k}{d_n+1}\right) = \frac{k^2\pi^2}{B^2} = \lambda_k\left(\prescript{B}{0}{\Delta_{\textnormal{dir}}}\right).
	\end{equation}
	Then it is immediate to prove that if
	$$
	\lambda_k\left(\prescript{B}{0}{\mathcal{L}_{\textnormal{BCs},V(y)}}\right) \neq  \frac{k^2\pi^2}{B^2} = \lambda_k\left(\prescript{B}{0}{\Delta_{\textnormal{dir}}}\right),
	$$
	then
	$$
	\lim_{n \to \infty} \left|\frac{(d_n+1)^2\omega^*_\eta\left(\frac{k}{d_n+1}\right)}{\lambda_k\left(\mathcal{L}_{\textnormal{BCs},p(x),q(x),w(x)}^{(n,\eta)} \right)} -1 \right| = \lim_{n \to \infty} \left|\frac{(d_n+1)^2\omega^*_\eta\left(\frac{k}{d_n+1}\right)}{\lambda_k\left(\mathcal{L}_{\textnormal{BCs},p(x),q(x),w(x)} \right)} -1 \right|=c_k>0.
	$$
	Moreover,
	$$
\lim_{n \to \infty}\lambda_k\left(\mathcal{L}_{\textnormal{BCs},p(x),q(x),w(x)}^{(n,\eta)} \right)=\lambda_k\left(\mathcal{L}_{\textnormal{BCs},p(x),q(x),w(x)} \right)=	\lambda_k\left(\prescript{B}{0}{\mathcal{L}_{\textnormal{BCs},V(y)}}\right) \sim \frac{k^2\pi^2}{B^2} \quad \mbox{for } k \to \infty,
	$$
	and then $c_k \to 0$ as $k \to \infty$.
\end{proof}

\begin{corollary}\label{cor:item_c}
	If
	\begin{equation*}
	\omega_\eta (x,\theta)= \frac{p(x)}{w(x)(b-a)^2}f_\eta(\theta), \qquad (x,\theta) \in [a,b]\times [0,\pi],
	\end{equation*}
	with $f_\eta(\theta)$ nonnegative, nondecreasing and such that $f_\eta(\theta) \sim \theta^2$ as $\theta\to 0$,
	then item (\ref{item:tilde_omega_asymptotics}) of Proposition \ref{prop:non_good_approximation} is satisfied.
\end{corollary}
\begin{proof}
	From \eqref{eq:rearrangment2} and \eqref{eq:rearrangment}, for all $t \in \left[0, t_0\right]$, with
	$$
	t_0 = (b-a)^{-2}\max_{[a,b]}\left(w(y)/p(y)\right)\sup_{[0,\pi]}\left(f_\eta(\theta)\right), 
	$$
	we have that 
	\begin{equation*}
	\phi_{\omega_\eta} (t) = \frac{1}{\pi(b-a)} \int_a^b \left( \int_0^\pi \mathbbm{1}_{\Omega_y(t)} (\theta) d\theta\right)dy,  \quad \mbox{with }\Omega_y(t):=\left\{\theta \in [0,\pi] \, : \, 0\leq f_\eta(\theta)\leq \frac{(b-a)^2w(y)}{p(y)}t\right\}.
	\end{equation*}
	By the monotonicity of $f_\eta$, it holds that $\Omega_y(t)=[0,\theta_y(t)]$, $\theta_y(t):=\sup \{\theta \in [0,\pi] \, : \, f_\eta(\theta)\leq \frac{(b-a)^2w(y)}{p(y)}t\}$, and then $\theta_y(t) \to 0$ as $t \to 0$. So, by the boundedness of $w(y)/p(y)$, for every $\epsilon>0$ there exists $\delta_\epsilon>0$ independent of $y$ such that for every $t\in [0,\min\{t_0;\delta_\epsilon\}]$ then $(1-\epsilon)\theta^2< f_\eta(\theta)< (1+\epsilon)\theta^2$, and 
	\begin{equation*}
	\frac{1}{\pi(b-a)} \int_a^b m\left( \Omega_y^{+}(t) \right)m(dy) \leq \phi_{\omega_\eta} (t) \leq \frac{1}{\pi(b-a)} \int_a^b m\left( \Omega_y^{-}(t) \right)m(dy),
	\end{equation*}
	with
	\begin{equation*}
	\Omega_y^{+}(t):=\left\{\theta \in [0,\pi] \, : \, \theta^2\leq \frac{(b-a)^2}{1+\epsilon}\frac{w(y)}{p(y)}t\right\}, \quad \Omega_y^{-}(t):=\left\{\theta \in [0,\pi] \, : \, \theta^2\leq \frac{(b-a)^2}{1-\epsilon}\frac{w(y)}{p(y)}t\right\}.
	\end{equation*}
	Then it holds that
	\begin{equation*}
	\frac{B}{\pi\sqrt{1+\epsilon}}\sqrt{t} \leq \phi_{\omega_\eta}(t) \leq \frac{B}{\pi\sqrt{1-\epsilon}}\sqrt{t}, \qquad B=\int_a^b \sqrt{\frac{w(y)}{p(y)}}dy.
	\end{equation*}
	By definition \eqref{eq:rearrangment}, $t\to 0$ as $x\to0$ and then
	\begin{equation*}
	(1-\epsilon)\frac{x^2\pi^2}{B^2}\leq \omega^*_\eta (x) \leq (1+\epsilon)\frac{x^2\pi^2}{B^2} \qquad \mbox{for }x \mbox{ small enough},
	\end{equation*}
	and the thesis follows.
\end{proof}

\begin{remark}\label{rem:not_good_approximation}
	The matrix methods of subsections \ref{ssec:FD}, \ref{ssec:IsoG} satisfy the hypothesis of Proposition \ref{prop:non_good_approximation}, see theorems \ref{thm:FD_symbol}, \ref{thm:Galerkin_symbol} and corollaries \ref{cor:FD_uniform}, \ref{cor:Galerkin_uniform}. Therefore, in general, a uniform sampling of their spectral symbols does not provide an accurate approximation of the eigenvalues $\lambda_k\left( \mathcal{L}_{\textnormal{BCs},p(x),q(x),w(x)}^{(n,\eta)}\right)$ and $\lambda_k\left(\mathcal{L}_{\textnormal{BCs},p(x),q(x),w(x)}\right)$, in the sense of the relative error. See subsections \ref{ssec:example_uniform_3_points},\ref{ssec:FD_nonuniform} and \ref{ssec:galerkin} for numerical examples. On the other hand, 
	$$
	\left\|  \omega^*_\eta\left(\frac{k}{n+1}\right)  - \lambda_k\left( \hat{\mathcal{L}}_{\textnormal{BCs},p(x),q(x),w(x)}^{(n,\eta)}\right)\right\|_\infty \to 0 \qquad \mbox{as } n\to \infty,
	$$
	if we exclude the outliers, see Corollary \ref{cor:discrete_Weyl_law} and Figure \eqref{fig:eig_symbol_comparison}.
\end{remark}

\section{$(2\eta+1)$-points central FD discretization}\label{ssec:FD}
Let us consider the following one dimensional self-adjoint operator
\begin{align}
&\mathcal{L}_{\textnormal{dir},p(x)}[u](x):= -\left(p(x)u'(x)\right)', \qquad \textnormal{dom}\left(\mathcal{L}_{\textnormal{dir},p(x)}\right)=W^{1,2}_0((a,b)), \label{eq:SLO1}
\end{align}
with $p\in C^2([a,b])$ and $p>0$ (see \cite{Davies,EM99}). About a general review of FD methods we refer to \cite{Smith85}. Fix $\eta,n\in \N$ and $\eta \geq 1$. Given a standard equispaced grid $\bfx=\left\{ x_j \right\}_{j=1-\eta}^{n+\eta}\subseteq \left[\bar{a},\bar{b}\right]$, with
$$
\bar{a}=a-(b-a)\frac{\eta-1}{n+1}<a,\quad \bar{b}=b + (b-a)\frac{n+\eta}{n+1}>b,\quad  x_j = a + (b-a)\frac{j}{n+1},
$$ 
let us consider a $C^1$-diffeomorphism $\tau : [a,b]\to [a,b]$ such that $\tau'(x)\neq 0$, $\tau(a)=a, \tau(b)=b$ and let us consider its piecewise $C^1$-extension $\bar{\tau}: \left[\bar{a}, \bar{b}\right] \to  \left[\bar{a}, \bar{b}\right]$ such that
\begin{equation}\label{eq:extending_tau}
\bar{\tau}(x) =\begin{cases}
x & \mbox{if } x \leq a,\\
\tau(x) &\mbox{if } x \in (a,b)\\
x & \mbox{if } x\geq b.
\end{cases}
\end{equation}
By means of the piecewise $C^1$-diffeomorphism $\bar{\tau}$ we have a new (extended) grid $\bar{\bfx}=\left\{ \bar{x}_j \right\}_{j=1-\eta}^{n+\eta}\subset [\bar{a},\bar{b}]$, non necessarily uniformly equispaced, with $\bar{x}_j = \bar{\tau}(x_j)$. Combining together the high-order central FD schemes presented in \cite{AS05,AS11,Li05}, we obtain the following matrix operator as approximation of \eqref{eq:SLO1}:

\begin{equation*}
\mathcal{L}^{(n,\eta)}_{\textnormal{dir},p(\bar{x})} = \begin{bmatrix}
l_{1,1} & \cdots & l_{1,1+\eta} & 0 & \cdots & 0\\
\vdots & l_{2,2} &  & l_{2,2+\eta}  &\cdots &0\\
l_{1+\eta,1}&  & \ddots & & \ddots & 0\\
& \ddots  &  & l_{n-\eta,n-\eta}& \ddots & l_{n-\eta,n}\\
0& & & & & \\
0&\cdots&0&l_{n,n-\eta}& \cdots & l_{n,n}
\end{bmatrix}\in \R^{n\times n} ,
\end{equation*}
where, if we define the $C^0$-extension of $p(x)$ to $[\bar{a},\bar{b}]$ as
\begin{equation*}
\bar{p}(x) = \begin{cases}
p(a) & \mbox{for } x\leq a,\\
p(x) & \mbox{for } x \in (a,b)\\
p(b) & \mbox{for } x\geq b,
\end{cases}
\end{equation*}
and the element $l_{i,j}$ as 
\begin{equation}\label{l_ij}
l_{i,j}=\frac{2\bar{p}\left(\frac{\bar{x}_j + \bar{x}_i}{2}\right)\sum_{\substack{m=i-\eta\\m\neq i,j}}^{i+\eta} \prod_{\substack{k=i-\eta\\k\neq i,j,m}}^{i+\eta}\left(\bar{x}_k - \bar{x}_i\right)}{\left(\bar{x}_j-\bar{x}_i\right)\prod_{\substack{k=i-\eta\\k\neq i,j}}^{i+\eta}\left(\bar{x}_k - \bar{x}_j\right)}, \qquad \mbox{for } i\neq j,\; \begin{cases}
i= 1,\ldots,n,\\
j=i-\eta,\ldots, i+\eta,
\end{cases}
\end{equation}
then the generic matrix element of $\mathcal{L}^{(n,\eta)}_{\textnormal{dir},p(\bar{x})}$ is given by
\begin{equation*}
\left(\mathcal{L}^{(n,\eta)}_{\textnormal{dir},p(\bar{x})}\right)_{i,j} = \begin{cases}l_{i,j} & \mbox{for }  i\neq j,\; |i-j|\leq\eta,\; i,j=1,\ldots,n,\\
-\sum_{\substack{k=i-\eta\\k\neq i}}^{i+\eta} l_{i,k} & \mbox{for } i=j,\; i=1,\ldots,n,\\
0 & \mbox{for } i\neq j,\; |i-j|>\eta,\; i,j=1,\ldots,n.
\end{cases}
\end{equation*}
The extended functions $\bar{\tau},\bar{p}$ on $[\bar{a},\bar{b}]\supset [a,b]$ serve to implement correctly the BCs.  With abuse of notation, we will call $\eta$ the \emph{order of approximation} of the central FD method. We have the following results.
\begin{theorem}\label{thm:FD_symbol}
	In the above assumptions, for $\eta\geq 1$ it holds that
	\begin{enumerate}[(i)]
		\item\label{item_spectral_conv_thm:FD_symbol} the eigenvalues $\lambda_k\left(\mathcal{L}^{(n,\eta)}_{\textnormal{dir},p(\bar{x})}\right)$ are real for every $k$ and
		\begin{equation*}
		\lim_{n \to \infty}\lambda_k\left(\mathcal{L}^{(n,\eta)}_{\textnormal{dir},p(\bar{x})}\right) = \lambda_k\left(\mathcal{L}_{\textnormal{dir},p(x)}\right) \qquad \mbox{for every fixed }k.
		\end{equation*}
		\item\label{item_spectral_symbol_thm:FD_symbol}
		\begin{equation}\label{eq:FD_symbol}
		\left\{ (n+1)^{-2}\mathcal{L}^{(n,\eta)}_{\textnormal{dir},p(\bar{x})} \right\}_n=\left\{ \hat{\mathcal{L}}^{(n,\eta)}_{\textnormal{dir},p(\bar{x})} \right\}_n \sim_{\lambda} \omega_{\eta} \left(x,\theta\right) \qquad (x,\theta) \in [a,b]\times[0,\pi], 
		\end{equation}
		where
		\begin{equation*}
		\omega_{\eta} \left(x,\theta\right)=\frac{p\left(\tau(x)\right)}{\left(\tau'(x)\right)^2(b-a)^2}f_{\eta}(\theta), 
		\end{equation*}
		and
		\begin{equation}\label{FD_coefficients}
		f_\eta (\theta) = d_{\eta,0} + 2\sum_{k=1}^\eta d_{\eta,k}\cos(k\theta), \qquad d_{\eta,k}= \begin{cases}
		(-1)^{k} \frac{\eta!\eta!}{(\eta-k)!(\eta+k)!}\frac{2}{k^2} & \mbox{for } k=1,\cdots \eta,\\
		-2\sum_{j=1}^\eta d_{\eta,j} &\mbox{for } k=0.
		\end{cases}
		\end{equation}
		\item\label{item_spectral_symbol_thm:no_outliers} $\lambda_k\left(\hat{\mathcal{L}}^{(n,\eta)}_{\textnormal{dir},p(\bar{x})}\right)\in R_{\omega_{\eta}}$ for every $k,n$.
	\end{enumerate}
\end{theorem}
\begin{remark}\label{rem:spatial_spectral_variables}
	The spectral symbol $\omega_\eta(x,\theta)$ consists of the product of two functions: 
	\begin{itemize}
		\item of the function $\frac{p\left(\tau(x)\right)}{\left(\tau'(x)\right)^2(b-a)^2}$ which consists of the diffusion coefficient $p(x)$ and the diffeomorphism $\tau(x)$, which are all intrinsic to the differential operator \eqref{eq:SLO1} itself and depend on the spatial variable $x$;
		\item of $f_\eta(\theta)$, which is intrinsic to the discretization method and depends on the spectral variable $\theta$.
	\end{itemize}
\end{remark}

\begin{proof}
	The proof of item \eqref{item_spectral_conv_thm:FD_symbol} is long and technical, and we avoid to present it here. Let us just mention that it can be proved by a straightforward generalization of standard techniques, see \cite[Theorem 1]{C69} and \cite{Gary65,CFL}. About item \eqref{item_spectral_symbol_thm:FD_symbol}, we recall that the \textquotedblleft hat\textquotedblright means that the matrix operator has been weighted by $(n+1)^2$. Let us preliminarily observe that in the case of $\tau(x)=x$ and $p(x)\equiv  1$ then $\hat{\mathcal{L}}^{(n,\eta)}_{\textnormal{dir}}$ is a symmetric Toeplitz matrix defined by
	$$
	\left(\hat{\mathcal{L}}^{(n,\eta)}_{\textnormal{dir}}\right)_{i,j} = \frac{d_{\eta,|i-j|}}{(b-a)^{2}} \quad \mbox{or equivalently}\quad \mathcal{L}^{(n,\eta)}_{\textnormal{dir}}= (b-a)^{-2}T_n\left(f_\eta\right) \quad \mbox{(see Section \ref{ssec:notations})},
	$$
	where $d_{\eta,k}$ are the coefficients of the trigonometric polynomial $f_\eta$ in \eqref{FD_coefficients}, see \cite[Corollary 2.2]{Li05} and \cite[Equation (27)]{KO99}. By standard results on the eigenvalues distribution of Toeplitz matrices (see for example \cite[\textbf{T 4} p. 168]{GS17}) it holds that 
	\begin{equation}\label{eq:1}
	\left\{\hat{\mathcal{L}}^{(n,\eta)}_{\textnormal{dir}}\right\}_n \sim_{\lambda} \frac{f_\eta (\theta)}{(b-a)^2}, \qquad \theta \in [0,\pi].
	\end{equation}
	The strategy to prove \eqref{item_spectral_symbol_thm:FD_symbol} is then the following:
	\begin{itemize}
		\item show that $\hat{\mathcal{L}}^{(n,\eta)}_{\textnormal{dir},p(\bar{x})} = X_n + Y_n$, with $X_n$ symmetric and such that $\|X_n\|,\|Y_n\|\leq C$, $n^{-1}\|Y_n\|_1\to 0$ as $n \to \infty$, where $\|\cdot\|$ and $\|\cdot\|_1$ are the spectral norm and the Schatten $1$-norm, respectively;
		\item show that $\{X_n\}_n \sim_{\textnormal{GLT}} \omega_{\eta} \left(x,\theta\right)$, see \cite[Definition 8.1]{GS17};
		\item Conclude that $\{\hat{\mathcal{L}}^{(n,\eta)}_{\textnormal{dir},p(\bar{x})}\}_n\sim_{\lambda} \omega_{\eta} \left(x,\theta\right)$ by \cite[\textbf{GLT 2} p. 170]{GS17}.
	\end{itemize} 
	Let us define $X_n=\hat{\mathfrak{L}}^{(n,\eta)}_{\textnormal{dir},p(\bar{x})}$, substituting the $l_{i,j}$ entries \eqref{l_ij} with 
	\begin{equation*}
	\mathfrak{l}_{i,j}:=\frac{\bar{p}\left(\frac{\bar{x}_j + \bar{x}_i}{2}\right)}{\left[\bar{\tau}'\left(\frac{x_j + x_i}{2}\right)\right]^2(b-a)^2}\cdot d_{\eta,|i-j|},
	\end{equation*}
	with the convention that in $a$ and $b$ we take the right and left limit of $\bar{\tau}'$, respectively. Then, $\hat{\mathfrak{L}}^{(n,\eta)}_{\textnormal{dir},p(\bar{x})}$ is symmetric and $\left\|\hat{\mathfrak{L}}^{(n,\eta)}_{\textnormal{dir},p(\bar{x})}\right\|\leq \max_{[a,b]}\left\{ |p(\tau(x))/\tau'(x)^2|\right\}\cdot\max_{[0,\pi]}\left\{|f_\eta (\theta)|\right\}$. Moreover, by the regularity of $p(x)$ and $\tau(x)$, it is not difficult to prove that
	\begin{equation*}
	\left(\hat{\mathcal{L}}^{(n,\eta)}_{\textnormal{dir},p(\bar{x})} - \hat{\mathfrak{L}}^{(n,\eta)}_{\textnormal{dir},p(\bar{x})}\right) = Y_n,
	\end{equation*} 
	such that 
	$$
	\|Y_n\|\leq c, \qquad n^{-1} \|Y_n\|_1 \to 0.
	$$
	If we finally show that $\{\hat{\mathfrak{L}}^{(n,\eta)}_{\textnormal{dir},p(\bar{x})}\}_n\sim_{\textnormal{GLT}}\omega_{\eta}(x,\theta)$, we can conclude. Define now
	$$
	S_n := \diag_{i=1,\ldots,n}\left\{\frac{p\left(\bar{x}_i\right)}{\left[\tau'\left(x_i\right)\right]^2}\right\}\hat{\mathcal{L}}^{(n,\eta)}_{\textnormal{dir}}.
	$$ 
	By \eqref{eq:1} and \cite[\textbf{GLT 3-4} p. 170]{GS17}, it holds that $\{S_n\}_n\sim_{\textnormal{GLT}} \omega_{\eta}(x,\theta)$. Again, by the regularity of $p$ and $\tau$, by direct calculation it is possible to show that $\left\|\hat{\mathfrak{L}}^{(n,\eta)}_{\textnormal{dir},p(\bar{x})} - S_n\right\|\to 0$ as $n\to \infty$, and therefore that $\{ \hat{\mathfrak{L}}^{(n,\eta)}_{\textnormal{dir},p(\bar{x})} - S_n \}_n\sim_{\textnormal{GLT}}\,0$ by \cite[\textbf{Z 2} p. 167]{GS17}. Then, by the GLT algebra \cite[\textbf{GLT 3-4} p. 170]{GS17} it follows that $\{\hat{\mathfrak{L}}^{(n,\eta)}_{\textnormal{dir},p(\bar{x})}\}_n\sim_{\textnormal{GLT}}\omega_{\eta}(x,\theta)$. Finally, item \eqref{item_spectral_symbol_thm:no_outliers} can be recovered by \cite[Theorem 2.2]{Serra00}.
\end{proof}

\begin{corollary}\label{cor:FD_uniform}
	For every fixed $\eta$, the function $f_\eta$ is differentiable, nonnegative, monotone increasing on $[0,\pi]$ and it holds that
	\begin{equation*}
	f_\eta(\theta)\sim \theta^2 \quad \mbox{as }\theta \to 0, \qquad \lim_{\eta \to \infty}\sup_{\theta \in [0,\pi]}\left|f_\eta(\theta) - \theta^2\right|=0.
	\end{equation*}
\end{corollary}
\begin{proof}
	$f_\eta(\theta)$ is obviously $C^\infty([0,\pi])$. Let us begin to prove that $f_\eta(\theta)\sim \theta^2$ as $\theta \to 0$ for every fixed $\eta$, and that it is monotone nonnegative on $[0,\pi]$. By the Taylor expansion at $\theta=0$ we get
\begin{align*}
f_\eta(\theta) 
&= d_{\eta,0} + \sum_{\substack{k=-\eta\\k\neq 0}}^\eta d_{\eta,k}\cos(k\theta)\\
&= d_{\eta,0} + \sum_{\substack{k=-\eta\\k\neq 0}}^\eta d_{\eta,k}\left(1 - \frac{(k\theta)^2}{2}\right) + o(\theta^3)\\
&=  -\theta^2\sum_{\substack{k=-\eta\\k\neq 0}}^\eta  (-1)^{k} \frac{\eta!\eta!}{(\eta-k)!(\eta+k)!} + o(\theta^3)\\
&= -\theta^2 \left[(-1)^\eta\frac{\eta!\eta!}{(2\eta)!}\sum_{\substack{m=0\\m\neq \eta}}^{2\eta} (-1)^{m}\binom{2\eta}{m}\right] + o(\theta^3)\\
&= \theta^2 + o(\theta^3).
\end{align*} 
Moreover, let us observe that
\begin{equation*}
f_\eta(\theta)=  d_{\eta,0}+ 2 \sum_{k=1}^\eta \left|d_{\eta,k}\right|\cos(k(-\theta+\pi)),\quad
f_\eta'(\theta)=  2 \sum_{k=1}^\eta k\left |d_{\eta,k}\right|\sin(k(-\theta+\pi)).
\end{equation*}
Define then 
$$
g (\psi) = \sum_{k=1}^\eta a_{k} \sin(k\psi), \qquad \mbox{with }\psi=-\theta+\pi \in [0,\pi], \qquad \mbox{and } a_{k} = 2k\left|d_{\eta,k}\right|.
$$
It is immediate to check that $a_1 \geq a_2 \geq \ldots \geq a_n > 0$ and that
$$
(2k)a_{2k} \leq (2k-1)a_{2k-1} \quad \forall k \geq 1.
$$	
By \cite[Theorem 1]{AS74} we can conclude that $g(\psi)>0$ on $(0,\pi)$ and then $f_\eta'(\theta)>0$ on $(0,\pi)$. Since $f_\eta(0)=0$, we deduce that $f_\eta(\theta)\geq 0$ on $[0,\pi]$. 

The second part of the thesis is an immediate consequence of identities \eqref{FD_coefficients}. Indeed, since
\begin{equation*}
\lim_{\eta \to \infty} d_{\eta,k}=(-1)^k \frac{2}{k^2} \quad\mbox{ and }\quad
\lim_{\eta \to \infty}\sum_{k=1}^{\eta} (-1)^k\left(\frac{2}{k^2} - \left|d_{\eta,k} \right|\right)=0,
\end{equation*}
we conclude that
$$
-\lim_{\eta\to \infty} 2\sum_{k=1}^\eta d_{\eta,k} = -2\sum_{k=1}^\infty (-1)^{k}\frac{2}{k^2}= \frac{\pi^2}{3}.
$$
For the same reasons, for every fixed $\theta \in [0,\pi]$ it holds that
\begin{equation*}
\lim_{\eta\to \infty} f_\eta(\theta) = \frac{\pi^2}{3} + 4\sum_{k=1}^\infty \frac{(-1)^k}{k^2}\cos(k\theta)= \theta^2,
\end{equation*}
being $\left\{\pi^2/3\right\}\cup \left\{(-1)^k4/k^2\right\}_{k\geq1}$ the Fourier coefficients of $\theta^2$ on $[0,\pi]$, and the convergence is then uniform.
\end{proof}

\section{Isogeometric Galerkin discretization by B-splines of degree $\eta$ and smoothness $C^{\eta-1}$}\label{ssec:IsoG}

For a general review of the IgA discretization method we refer the reader to \cite{CHB,HER}, so we skip all the introductions and we present directly some known results.

\begin{theorem}\label{thm:Galerkin_symbol}
	Let $(a,b)$ be discretized by a uniform mesh $\left\{ x_j \right\}_{j=1}^n$ of step-size $(n+1)^{-1}$ and let $\tau : [a,b]\to [a,b]$ a $C^1$-diffeomorphism such that $\tau(a)=a,\tau(b)=b$ and $\tau'(x)\neq 0$ for every $x \in [a,b]$. Let $\bar{x}:=\tau(x)$, $\eta\geq 1$, and let us indicate with $\mathcal{L}^{(n,\eta)}_{\textnormal{dir},p(\bar{x})}$ the discrete operator obtained from \eqref{eq:SLO1} by an IgA discretization scheme with B-splines of degree $\eta$ and  smoothness $C^{\eta-1}$. Then
	\begin{enumerate}[(i)]
		\item\label{spect_conv_IgA} the eigenvalues $\lambda_k\left(\mathcal{L}^{(n,\eta)}_{\textnormal{dir},p(\bar{x})}\right)$ are real for every fixed $k$ and
		$$
		\lim_{n \to \infty}\lambda_k\left(\mathcal{L}^{(n,\eta)}_{\textnormal{dir},p(\bar{x})}\right)= \lambda_k\left(\mathcal{L}_{\textnormal{dir},p(x)}\right);
		$$
		\item\label{symb_IgA} it holds that
		$$ 
		\left\{(n+1)^{-2}\mathcal{L}^{(n,\eta)}_{\textnormal{dir},p(\bar{x})}\right\}_n=\left\{\hat{\mathcal{L}}^{(n,\eta)}_{\textnormal{dir},p(\bar{x})}\right\}_n\sim_{\lambda}\omega_{\eta}(x,\theta), \qquad (x,\theta) \in [a,b]\times [0,\pi],
		$$ 
		where
		\begin{align}
		\omega_\eta( x,\theta)& =\frac{p(\tau(x))}{\left(\tau'(x)\right)^2(b-a)^2}\,f_\eta(\theta),\label{e-symbol}
		\end{align}
	\end{enumerate}
	For  example, for $\eta=1,2$, $f_\eta(\theta)$ has the following analytical expressions
	\begin{equation*}
	f_1(\theta)=\frac{6(1-\cos\theta)}{2+\cos\theta},\qquad
	f_2(\theta)=\frac{20(3-2\cos\theta-\cos(2\theta))}{33+26\cos\theta+\cos(2\theta)}.
	\end{equation*}
\end{theorem}%
\begin{proof}
	For item \eqref{spect_conv_IgA} see \cite{BBCHS06,PDC}. For item \eqref{symb_IgA}, see \cite[Theorem 10.15]{GS17}.
\end{proof}

About the spectral symbol $\omega_{\eta}$ see again Remark \ref{rem:spatial_spectral_variables}. We have an analogue of Corollary \ref{cor:FD_uniform}.
\begin{corollary}\label{cor:Galerkin_uniform}
	For every fixed $\eta$,	the function $f_\eta$ is differentiable, nonnegative, monotone increasing on $[0,\pi]$ and it holds that
	\begin{equation*}
	f_\eta(\theta)\sim \theta^2 \quad \mbox{as }\theta\to 0, \qquad \lim_{\eta \to \infty}\sup_{\theta \in [0,\pi]}\left|f_\eta(\theta) - \theta^2\right|=0.
	\end{equation*}
\end{corollary}
\begin{proof}
	See \cite[Theorem 1, Theorem 2 and Lemma 1]{EFGMSS18}.
\end{proof}

With abuse of notation, we will call $\eta$ the \emph{order of approximation} of the IgA method.

\end{document}